\documentclass[12pt,english]{article}

\usepackage{hyperref}
\hypersetup{
    unicode      = false,
    pdftoolbar   = true,
    pdfmenubar   = true,
    pdffitwindow = true,
    pdfnewwindow = true,
    colorlinks   = true,
    linkcolor    = blue,
    citecolor    = blue,
    filecolor    = blue,
    urlcolor     = blue,
    breaklinks   = true
}
\usepackage{stmaryrd}
\SetSymbolFont{stmry}{bold}{U}{stmry}{m}{n}
\usepackage[english]{babel}
\usepackage[utf8]{inputenc}

\usepackage{tcolorbox}

\usepackage{amsmath}
\usepackage{amsthm}
\usepackage{amssymb}
\usepackage{multirow}
\usepackage{epsfig}
\usepackage{soul}
\usepackage{colortbl}
\usepackage{curves}
\usepackage{epic}
\usepackage{pgf}
\usepackage{rotating}
\usepackage{array}
\usepackage[margin=1in]{geometry}
\usepackage{times}
\usepackage{xspace}
\usepackage{xcolor}
\usepackage{mathtools}
\usepackage[nameinlink,capitalise,noabbrev]{cleveref}
\usepackage{booktabs}
\usepackage{mathtools}
\usepackage{graphicx}
\usepackage{pgfplots} 
\pgfplotsset{compat=1.18} 
\usepgfplotslibrary{groupplots}

\usepackage[numbers]{natbib}

\Crefname{subsection}{Section}{Sections}
\crefname{assumption}{Assumption}{Assumptions}
\Crefname{assumption}{Assumption}{Assumptions}
\usepackage{graphicx}
\usepackage{tikz}
\usetikzlibrary{shapes,arrows}
\usetikzlibrary{positioning,fit,backgrounds}
\usetikzlibrary{calc}
\usetikzlibrary{patterns}
\usetikzlibrary{intersections}
\usepackage[export]{adjustbox}
\usetikzlibrary{shapes.geometric}
\usetikzlibrary{shapes.misc}
\tikzset{cross/.style={cross out, draw=black, minimum size=2*(#1-\pgflinewidth), inner sep=0pt, outer sep=0pt},
cross/.default={1pt}}
\usetikzlibrary{patterns.meta}

\pgfdeclarepattern{
  name = random dots,
  type = uncolored,
  bottom left = \pgfpoint{0pt}{0pt},
  top right = \pgfpoint{8pt}{8pt},
  tile size = \pgfpoint{8pt}{8pt},
  code = {
    \pgfpathcircle{\pgfpoint{1pt}{6pt}}{0.4pt}
    \pgfusepath{fill}
    \pgfpathcircle{\pgfpoint{3pt}{2pt}}{0.35pt}
    \pgfusepath{fill}
    \pgfpathcircle{\pgfpoint{6pt}{5pt}}{0.45pt}
    \pgfusepath{fill}
    \pgfpathcircle{\pgfpoint{7pt}{1pt}}{0.3pt}
    \pgfusepath{fill}
  }
}

\usepackage{float}
\usepackage{subcaption}

\usepackage{accents}
\DeclareRobustCommand{\esh}{\accentset{\circ}{\mathbb{R}}^n_k}
\DeclareRobustCommand{\eshzero}{\accentset{\circ}{\mathbb{R}}^n_0}
\DeclareRobustCommand{\eshkplusun}{\accentset{\circ}{\mathbb{R}}^n_{k+1}}

\def\R{{\mathbb{R}}}
\def\C{{\mathcal{C}}}
\def\N{{\mathbb{N}}}

\def\K{{\mathcal{K}}}

\def\L{{\mathcal{L}}}
\def\compact{{L}}
\def\W{{\mathcal{W}}}
\def\V{{\mathbb{V}}}
\def\P{{\mathbb{P}}}

\def\deltA{{ \delta}}
\def\DeltA{{ \Delta}}

\def\increaset{{\tt  increase}}
\def\decreaset{{\tt  decrease}}

\def\T{{T^\text{H}_{\Omega}}}
\def\v{{\hat v}}

\def\nunkp1{{n^{k+1}_\text{U}}}

\newcommand{\argmin}{\mathop{\mathrm{argmin}}}
\newcommand{\nomad}{{\tt NOMAD}\xspace}
\def\mads{\textnormal{MADS}\xspace}
\def\ads{\textnormal{ADS}\xspace}
\def\madspb{\textnormal{MADS-PB}\xspace}
\def\adspb{\textnormal{ADS-PB}\xspace}

\newcommand{\solar}[1]{\textnormal{\textsf{SOLAR #1}}\xspace}
\newcommand{\styrene}{\textnormal{\textsf{STYRENE}}\xspace}

\newtheorem{theorem}{Theorem}[section]

\newtheorem{assumption}{Assumption}

\newtheorem{definition}{Definition}
\crefname{definition}{Definition}{Definitions}
\Crefname{definition}{Definition}{Definitions}

\definecolor{Red}{rgb}{1,0,0}
\definecolor{Green}{rgb}{0,.6,0}
\definecolor{Blue}{rgb}{0,0,1}

\usepackage{algorithm}             
\usepackage[commentColor=blue]{algpseudocodex}
\usepackage{tcolorbox}
\tcbuselibrary{skins,breakable}
\usepackage{xparse}
\usepackage{xparse}
\usepackage{tcolorbox}
\tcbuselibrary{skins,breakable}
\usepackage[commentColor=blue]{algpseudocodex}

\newcounter{algobox}[section]

\usepackage[textwidth=22mm, backgroundcolor=yellow, linecolor=orange, textsize=scriptsize]{todonotes}

\newcommand{\BoxRef}[1]{%
\hyperref[#1]{\fbox{Box~\ref*{#1}}}%
}

\title{
{Adaptive direct search algorithms  \\ with relaxable and quantifiable constraints}
\thanks{
{GERAD}
and D\'epartement de math\'ematiques et g\'enie industriel,
Polytechnique Montr\'eal,
C.P.~6079, Succ. Centre-ville,
Montr\'eal, Qu\'ebec, Canada H3C~3A7.   }

\author{
\href{mailto:Charles.Audet@gerad.ca}{Charles Audet}\thanks{
\href{mailto:Charles.Audet@gerad.ca}{\url{charles.audet@polymtl.ca} \hfill https://orcid.org/0000-0002-3043-5393}
}
\and
\href{mailto:theo.denorme@etud.polymtl.ca}{Th\'eo Denorme}\thanks{
\href{mailto:theo.denorme@etud.polymtl.ca}{\url{theo.denorme@etud.polymtl.ca} \hfill https://orcid.org/0009-0008-5947-4071}
} 
\and \href{mailto:youssef.diouane@polymtl.ca}{Youssef Diouane}\thanks{
\href{mailto:youssef.diouane@polymtl.ca}{\url{youssef.diouane@polymtl.ca} \hfill https://orcid.org/0000-0002-6609-7330}
} \and \href{mailto:ledigabel.sebastien@polymtl.ca}{S\'ebastien Le~Digabel}\thanks{
\href{mailto:ledigabel.sebastien@polymtl.ca}{\url{sebastien.le-digabel@polymtl.ca} \hfill https://orcid.org/0000-0003-3148-5090}
} \and \href{mailto:ledigabel.sebastien@polymtl.ca}{Christophe Tribes}\thanks{
\href{mailto:tribes.christophe@polymtl.ca}{\url{christophe.tribes@polymtl.ca} \hfill https://orcid.org/0000-0002-8740-6155}
} 
}}

\begin{document}
\maketitle


\paragraph{Abstract}
This work introduces \adspb, an extension of the Adaptive Direct Search (\ads) framework for solving  constrained blackbox optimization problems.
With \ads, iterates progress without relying on mesh structures or sufficient decrease conditions on the objective function value.
Unlike the extreme barrier approach used in \ads, where only unrelaxable constraints are considered, the proposed method also handles quantifiable and relaxable constraints using a Progressive Barrier (PB) mechanism that exploits both constraint and objective function values. A convergence analysis of the proposed framework under mild assumptions is presented. 
The performance of the proposed method is assessed using sets of analytical and simulation-based constrained test problems and is compared with state-of-the-art blackbox optimization solvers, including the PB approach within the Mesh Adaptive Direct Search (\mads) framework.


\paragraph{Keywords}
Derivative-free optimization;
Constrained optimization;
Adaptive direct search; Relaxable constraints. 

\paragraph{AMS subject classifications}
90C30, 90C56, 49J52.

\section{Introduction}

Constrained blackbox optimization arises in many applications where both the objective function and the constraints can only be accessed through evaluations, often at a significant computational cost and without derivative information. In such settings, the optimization framework must not only remain robust in the absence of smoothness assumptions but also be flexible enough to exploit informative trial points proposed by search heuristics, surrogate models, or problem-specific strategies.
The present work considers constrained blackbox optimization problems~\cite{AuHa2026} of the form
\begin{equation}
\label{eq-pb}
      \min_{x \in \Omega} f(x)
\end{equation}
where the feasible set is defined as $\Omega = \left\{ x \in X: c_j(x) \leq 0 \text{ for all } j \in J \right\}$, and $X \subseteq \R^n$ denotes the blackbox domain.
Using the taxonomy defined in~\cite{LedWild2015}, the set  $X$ incorporates the \emph{a priori known}~\cite{LedWild2015} constraints of the problem, such as bound constraints or other explicitly defined restrictions. The index set $J$ is finite and identifies the inequality constraints.
Both the objective function $f\colon X \to \overline{\R}$ and the constraint functions $\{c_j\colon X \to \overline{\R}\}_{j \in J}$ are assumed to be available only through blackbox evaluations. The use of the extended real line $\overline{\R}$ instead of $\R$ reflects the fact that, in blackbox settings, some evaluations may fail or return invalid outputs. For instance, a simulation may crash, a solver may not converge, or a function call may return an error at some trial points. In such cases, assigning the value $+\infty$ to the objective and/or constraint functions provides a convenient way to model these failed evaluations while keeping the optimization framework robust to this type of numerical issue.
The constraints are assumed to be \emph{quantifiable} and \emph{relaxable}~\cite{LedWild2015}.
They are said to be \emph{quantifiable} because each constraint function $c_j$, $j \in J$, provides a real-valued measure of violation, allowing one to assess not only feasibility but also the magnitude of infeasibility. Such constraints are said to be \emph{relaxable} because the objective may still be evaluated at infeasible points during the optimization process and used to guide the search toward feasibility, rather than being strictly excluded from the search.

\subsection{Motivation}
The present work extends the Adaptive Direct Search (\ads) framework of~\cite{denorme-ads-2025} to the class of constrained blackbox optimization problems of the form~\eqref{eq-pb}. While \ads relies on an extreme barrier~\cite{AuDe2006} to enforce feasibility, such an approach can be overly restrictive when constraint violations can be quantified and meaningfully reduced during the optimization process. 
This naturally motivates the use of a Progressive Barrier (PB) mechanism~\cite{AuDe09a}, which allows infeasible points to contribute useful information while gradually driving the optimization process toward feasibility.

More broadly, many established Derivative-Free Optimization (DFO) methods have been proposed for constrained problems, including evolution strategies~\cite{DIOUANE2021100001,cmaes}, a merit function approach for direct search~\cite{GrVi2014}, a mixed interior point method for direct search~\cite{BrCuLiSi2024}, and genetic algorithms~\cite{Deb2000}, among many others discussed in~\cite{AlAuGhKoLed2020,AuHa2026,CoScVibook,CuScVi2017,KoLeTo03a,LaMeWi2019}.
Furthermore, recent developments in DFO increasingly emphasize the use of surrogate and quadratic models to improve the quality of trial points and better exploit previously evaluated information. For instance, modern approaches combine quadratic interpolation with structured trust-region frameworks or subspace techniques to address high-dimensional or partially separable problems~\cite{Cartis2026, LiuLi2025}. Other works integrate adaptive sampling and machine learning models to improve surrogate accuracy and search efficiency~\cite{Gupta2024,Karantoumanis2024}. 
Recent model-and-search frameworks explicitly combine direct~search strategies with quadratic model construction to enhance local convergence properties while maintaining robustness in blackbox settings~\cite{Ma2025}.
These developments highlight the importance of optimization frameworks that can naturally incorporate model-based and heuristic search mechanisms without degrading the quality of the proposed trial points. 
More generally, related direct~search convergence mechanisms based on stepsize or radius updates are surveyed in~\cite{DzRiRoZe2025}.

\subsection{Contributions}
This work introduces \adspb, a PB extension of the \ads framework, designed to preserve the flexibility of \ads while allowing infeasible evaluations to be incorporated throughout the optimization process. The proposed framework is intended for blackbox settings where constraint violations carry meaningful information and can be exploited to guide the search, rather than being discarded entirely.

A central aspect of \adspb is that it removes two structural features that often limit the effective use of sophisticated search mechanisms in direct~search methods, namely the reliance on a mesh and the use of sufficient decrease conditions. By lifting these restrictions, the framework can directly exploit high-quality trial points generated by search heuristics, surrogate models, or hybrid strategies without altering them through mesh projection or rejecting them because of restrictive acceptance requirements. This makes \adspb particularly well suited to modern model-assisted DFO.

Beyond the algorithmic framework itself, this work also establishes a convergence analysis under standard assumptions of DFO, providing theoretical support for the proposed approach. Its practical relevance is further illustrated through an implementation within the \nomad software~\cite{nomad4paper} and through computational experiments on constrained analytical and blackbox test problems, where the behavior of \adspb is compared with that of existing DFO methods.

\subsection{Outline}

\Cref{sec:PB} recalls the key notions of the PB framework and fixes the terminology used throughout the paper.
\Cref{sec:ADS} presents the \adspb algorithmic framework.
The convergence analysis is given in \cref{sec:convergence}, and computational results, together with implementation aspects in \nomad, are reported in \cref{sec:numerical}.
Concluding remarks are provided in \cref{sec-discussion}.



\section{The progressive barrier approach}
\label{sec:PB}

A simple way to handle relaxable and quantifiable constraints is the two-phase extreme barrier approach:
The first phase ignores the objective function and attempts to find a feasible point by minimizing a measure of constraint violation.
The second phase focuses on reducing the objective function over the feasible set. 
This approach, although straightforward and simple to implement, might discard useful information from objective function evaluations at infeasible points.
The PB approach~\cite{AuDe09a} is designed to unify these two phases by allowing the algorithm to improve both feasibility and the objective while progressively restricting the admissible level of constraint violation.

\subsection{Constraint violation function and incumbent solutions}
\label{sub:incumbent}

The PB is a framework that may be combined with different direct~search mechanisms: it can be coupled with a mesh structure, as in~\cite{AuDe09a}, but it also admits alternative realizations, for instance, merit-function-based approaches that balance objective decrease and feasibility improvement~\cite{AuCoLedPey2016}.
Constraints violations are aggregated using the {\em constraints violation function} $h:\R^n \to \overline{\R}^+$ defined by
\begin{equation*}
\label{eq:constraint-violation}
h(x)\ =\ 
\begin{cases}
\displaystyle \sum_{j\in J}\max(0,c_j(x))^2, & \text{if } x\in X,\\
+\infty, & \text{otherwise}.
\end{cases}
\end{equation*}
With this notation, the feasible set may be written as  $\Omega = \{x\in X:\ h(x)=0\}$.

The main idea behind the PB approach is to use a \emph{barrier threshold} $h^k_{\max}$ on the constraint violation function at each iteration $k$. 
The barrier threshold is initialized to $+\infty$. 
As the algorithm is deployed, when a candidate solution has a constraint violation function value that exceeds the current threshold, it is considered inadmissible for replacing one of the incumbents. 
The threshold is then progressively decreased, making the algorithm increasingly selective with respect to feasibility.
Candidate solutions are compared against the current incumbents, which may be updated between successive iterations. Such comparisons are based either solely on the objective function value when the candidate is feasible or on both the objective function and constraint violation function values when the candidate is infeasible, as follows:
\begin{itemize}
\item
A feasible point $x \in \Omega$ is preferred over $y \in \Omega$ (denoted by $x \prec_{\mathrm{feas}} y$),  if it improves the value of $f$ (i.e., $f(x) < f(y)$).
\item
An infeasible point $x \in X \setminus \Omega$ is preferred over $y \in X \setminus \Omega$ (denoted by $x \prec_{\mathrm{inf}} y$) if it improves at least one of $f$ or $h$ without worsening the other, i.e., either $f(x) < f(y)\ \text{and}\ h(x) \le h(y) \text{, or,  }\ f(x) \le f(y)\ \text{and}\ h(x) < h(y)$.
\end{itemize}
In both situations, 
    if $x \prec_{\mathrm{feas}} y$ or $x \prec_{\mathrm{inf}} y$, 
    then $x$ is said to \emph{dominate} $y$, the notation is abbreviated by $x \prec y$.
The PB approach stores and orders the list of all trial points that were evaluated.
Let $\C^k$, referred to as the \emph{cache}, denote the set of all trial points where the objective and constraint functions were evaluated up to the start of iteration $k \in \N$.
Any point in $\C^k$
is called a \emph{visited point}.
The algorithm maintains up to two types of incumbent solutions.
The set of {\em feasible incumbents} is defined as 
\begin{equation}
\label{eq:Fk}
F^k \ :=\ \argmin_{x\in\C^k}\{f(x)\ :\  h(x)=0\}.
\end{equation}
This set is empty when the cache $\C^k$ does not contain any feasible point.
The second set of incumbents relies on the barrier threshold $h^k_{\max}$ and 
 on the set of infeasible non-dominated points, i.e.,
\begin{equation*}
\label{eq:Uk}
U^k := \Bigl\{ x\in \C^k \setminus \Omega \ : \text{ there exists no } y\in \C^k\setminus \Omega \text{ such that }\ y \prec_{\mathrm{inf}} x\Bigr\}.
\end{equation*}
The set of {\em infeasible incumbents} is defined as
\begin{equation}
\label{eq:Ik}
I^k \ :=\ \argmin_{x\in U^k}\{f(x) \ :\  0<h(x)\le h^k_{\max}\}.
\end{equation}
With these definitions,
$F^k$ contains the feasible points visited so far with the least value of $f$,
 and $I^k$ contains the infeasible non-dominated points visited so far 
 (whose constraint violation function value is less than the barrier threshold) that have the least value of $f$.

These two sets of incumbents ensure that the algorithm reports a meaningful incumbent: either a best feasible point when feasibility has been reached, or a best infeasible point when it has not.
To summarize the information carried by these sets, the framework associates up to two representative incumbent points at iteration $k$. If $F^k\neq\emptyset$, \textit{a feasible incumbent} $x^k_{\mathrm{F}}\in F^k$ is selected, and its value is denoted by $f_{\mathrm{F}}^k$, i.e.,
\begin{equation} \label{eq:xkF}
 x^k_{\mathrm{F}}\in F^k  \quad \mbox{and} \quad  f_{\mathrm{F}}^k \ := \ f\!\left(x^k_{\mathrm{F}}\right).
\end{equation}
If $I^k\neq\emptyset$ is selected, an \textit{infeasible incumbent} $x^k_{\mathrm{I}}\in I^k$ is chosen, together with its pair of values $(f_{\mathrm{I}}^k,h_{\mathrm{I}}^k)$, i.e.,
\begin{equation} \label{eq:xkI}
   x^k_{\mathrm{I}}\in I^k \quad \mbox{and} \quad (f_{\mathrm{I}}^k,h_{\mathrm{I}}^k)
\ :=\ \bigl(f(x^k_{\mathrm{I}}),\,h(x^k_{\mathrm{I}})\bigr).
\end{equation}
In the case where $F^k$ or $I^k$ is not a singleton, the incumbents $x^k_{\mathrm{F}} \in F^k$ and $x^k_{\mathrm{I}} \in I^k$ may be chosen arbitrarily. These incumbents will serve as reference points to determine the outcome of iteration $k$: a new trial point at which $f$ is evaluated is compared against $x^k_{\mathrm{F}}$ if it is feasible and against $x^k_{\mathrm{I}}$ if it is infeasible. 

\subsection{Dominating, improving and unsuccessful iterations}
During iteration $k$, trial points may improve the current incumbents in different ways.
A first possibility is through a \emph{dominating} point. 
This includes the case of a feasible trial point whose objective function value is strictly less than $f^k_{\mathrm{F}}$, when $F^k \neq \emptyset$, or any feasible trial point when $F^k = \emptyset$. 
Similarly, an infeasible trial point is considered dominating if it dominates $x^k_{\mathrm{I}}$ when $I^k \neq \emptyset$, or if it is the first infeasible point found when $I^k = \emptyset$. 
In all these cases, the iteration is declared \emph{dominating}.

All iterations do not produce dominating points, especially before feasibility is reached. 
In that case, the PB still allows the algorithm to make progress by focusing on feasibility: in the situations where iteration $k$ is not dominating
, the algorithm may accept an \emph{improving} point, that is, an infeasible visited point with a strictly smaller constraint violation function value than that of the current infeasible incumbent:
an {\em improving point} $x\in X$ satisfies the following inequalities:
\begin{equation*}
    0\ <\ h(x)\ <\ \!h^k_\text{I}. 
\end{equation*}
Any dominating or improving point is called a \textit{successful point}.
Finally, an iteration that is neither dominating nor improving is said to be {\em unsuccessful}.

\Cref{fig:hf} illustrates the three iteration types in the $(f,h)$-plane. 
The feasible incumbent $x^k_{\mathrm{F}}$ lies on the vertical axis $h=0$, while the infeasible incumbent $x^k_{\mathrm{I}}$ has coordinates $(f_{\mathrm{I}}^k,h_{\mathrm{I}}^k)$ with $h_{\mathrm{I}}^k>0$. The dominance relation partitions the plane into three regions.

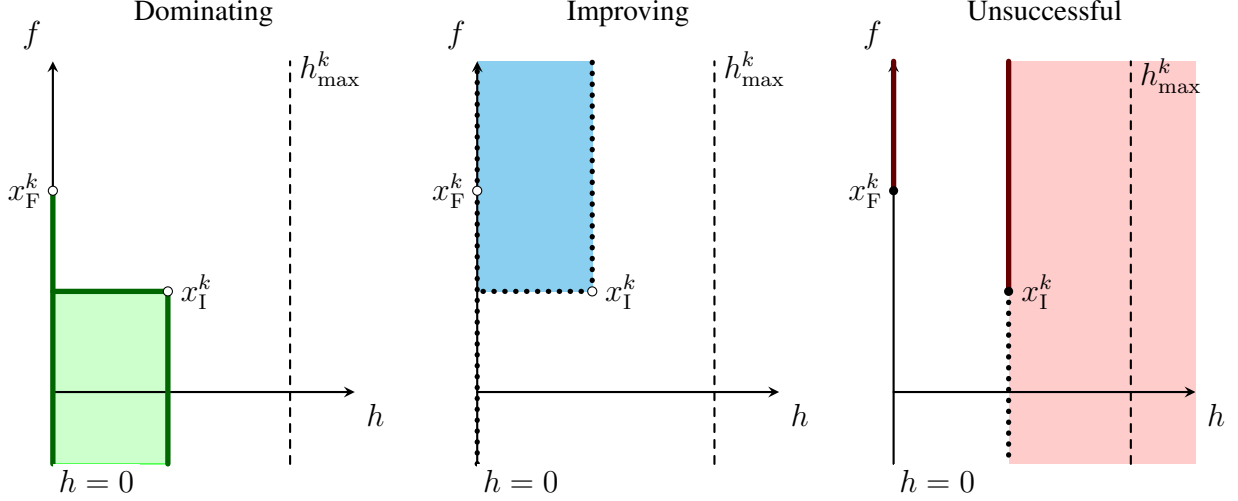
\begin{figure}[htb!]
\centering

\setlength{\tabcolsep}{2pt}
\newcommand{\HFsetup}{%
  \def\xmax{4.2}%
  \def\ymax{4.6}%
  \def\ymin{-1}%
  \def\hI{1.6}%
  \def\fF{2.8}%
  \def\fI{1.4}%
  \def\hmax{3.3}%
  \def\xmid{5.0}%
  \def\scale{1}
}
\begin{tabular}{ccc}

\HFsetup
\begin{tikzpicture}[
    >=stealth,
    line cap=round,
    line join=round,
    scale=0.95\scale
]

\node[anchor=north, font=\small] at ({0.5*\xmax}, {\ymax+1}) {Dominating};

\draw[green, fill=green!20] (0,\ymin) rectangle (\hI,\fI);

\draw[dashed, thick] (\hmax,\ymin) -- (\hmax,\ymax);
\node[
    fill=white,
    fill opacity=0.9,
    text opacity=1,
    inner sep=2pt,
    anchor=west
] at (\hmax+0.05,\ymax-0.1) {$h^k_{\max}$};

\draw[->, thick] (0,\ymin) -- (0,\ymax) node[above left] {$f$};
\draw[->, thick] (0,0) -- (\xmax,0) node[below right] {$h$};

\draw[line width=2pt,green!40!black] (0,\ymin) -- (0,\fF);
\draw[line width=2pt,green!40!black] (\hI,\fI) -- (0,\fI);
\draw[line width=2pt,green!40!black] (\hI,\fI) -- (\hI,\ymin);
\node[
    anchor=north west,
    fill=white,
    inner sep=2pt
] at (0,\ymin) {$h=0$};

\coordinate (xF) at (0,\fF);
\coordinate (xI) at (\hI,\fI);

\filldraw[fill=white] (xF) circle (1.8pt);
\node[left] at (xF) {$x^k_{\mathrm{F}}$};

\filldraw[fill=white] (xI) circle (1.8pt);
\node[
    fill=white,
    fill opacity=0.9,
    text opacity=1,
    inner sep=2pt,
    anchor=west
] at (\hI+0.1,\fI) {$x^k_{\mathrm{I}}$};

\end{tikzpicture}

&
\HFsetup
\definecolor{babyblue}{rgb}{0.54, 0.81, 0.94}
\begin{tikzpicture}[
    >=stealth,
    line cap=round,
    line join=round,
    scale=0.95\scale
]

\node[anchor=north, font=\small] at ({0.5*\xmax}, {\ymax+1}) {Improving};



\draw[ fill=babyblue, draw=none] (0,\fI) rectangle (\hI,\ymax);

\draw[dashed, thick] (\hmax,\ymin) -- (\hmax,\ymax);
\node[
    fill=white,
    fill opacity=0.9,
    text opacity=1,
    inner sep=2pt,
    anchor=west
] at (\hmax+0.05,\ymax-0.1) {$h^k_{\max}$};

\draw[->, thick] (0,\ymin) -- (0,\ymax) node[above left] {$f$};
\draw[->, thick] (0,0) -- (\xmax,0) node[below right] {$h$};

\draw[line width=2pt, dash pattern=on 0pt off 4pt] (\hI,\fI) -- (\hI,\ymax);
\draw[line width=2pt, dash pattern=on 0pt off 4pt] (\hI,\fI) -- (0,\fI);
\draw[line width=2pt, dash pattern=on 0pt off 4pt] (0,\ymin) -- (0,\ymax);

\node[
    anchor=north west,
    fill=white,
    inner sep=2pt
] at (0,\ymin) {$h=0$};

\coordinate (xF) at (0,\fF);
\coordinate (xI) at (\hI,\fI);

\filldraw[fill=white] (xF) circle (1.8pt);
\node[left] at (xF) {$x^k_{\mathrm{F}}$};

\filldraw[fill=white] (xI) circle (1.8pt);
\node[
    fill=white,
    fill opacity=0.9,
    text opacity=1,
    inner sep=2pt,
    anchor=west
] at (\hI+0.1,\fI) {$x^k_{\mathrm{I}}$};

\end{tikzpicture}

&
\HFsetup
\begin{tikzpicture}[
    >=stealth,
    line cap=round,
    line join=round,
    scale=0.95\scale
]

\node[anchor=north, font=\small] at ({0.5*\xmax}, {\ymax+1}) {Unsuccessful};


\draw[ fill=red!20, draw=none] (\hI,\ymin) rectangle (\xmax,\ymax);

\draw[dashed, thick] (\hmax,\ymin) -- (\hmax,\ymax);
\node[
    fill=white,
    fill opacity=0,
    text opacity=1,
    inner sep=2pt,
    anchor=west
] at (\hmax+0.05,\ymax-0.2) {$h^k_{\max}$};

\draw[->, thick] (0,\ymin) -- (0,\ymax) node[above left] {$f$};
\draw[->, thick] (0,0) -- (\xmax,0) node[below right] {$h$};

\draw[line width=2pt,red!40!black] (0,\ymax) -- (0,\fF);
\draw[line width=2pt,red!40!black] (\hI,\fI) -- (\hI,\ymax);
\draw[line width=2pt, dash pattern=on 0pt off 4pt] (\hI,\fI) -- (\hI,\ymin);

\node[
    anchor=north west,
    fill=white,
    inner sep=2pt
] at (0,\ymin) {$h=0$};

\coordinate (xF) at (0,\fF);
\coordinate (xI) at (\hI,\fI);

\fill (xF) circle (1.8pt);
\node[left] at (xF) {$x^k_{\mathrm{F}}$};

\fill (xI) circle (1.8pt);
\node[
    fill=white,
    fill opacity=0,
    text opacity=1,
    inner sep=2pt,
    anchor=west
] at (\hI+0.1,\fI) {$x^k_{\mathrm{I}}$};

\end{tikzpicture}

\end{tabular}

\caption{Partition
of the $(h,f)$-plane induced by the infeasible incumbent $x^k_{\mathrm{I}}$, the feasible incumbent $x^k_{\mathrm{F}}$, and the barrier threshold $h^k_{\max}$. Solid bold lines and filled points are included in the shaded region, whereas dotted bold lines are excluded.
}
\label{fig:hf}
\end{figure}

The \emph{dominating} region consists of points that dominate the current reference (either feasible points on $h=0$ with $f<f_{\mathrm{F}}^k$, highlighted in green on the leftmost figure, or infeasible points with $h\le h^k_{\mathrm{I}}$ and $f\le f^k_{\mathrm{I}}$ with at least one strict inequality). 
The \emph{improving} region corresponds to infeasible points that strictly decrease the violation, $0<h<h^k_{\mathrm{I}}$, even if $f$ is not improved.
It is illustrated in blue
in the central figure.
Finally, points that fall outside these two regions are \emph{unsuccessful} for iteration $k$.
The region is delimited in red in the rightmost figure.
The barrier threshold $h^k_{\max}$ appears as a vertical line: it indicates the current admissible range of constraint violation for selecting the infeasible incumbents, and it helps visualize how the method becomes more and more restrictive as it progresses.

The barrier threshold parameter $h^k_{\max}$ monotonically evolves 
depending on whether the iteration is dominating, improving, or unsuccessful. The update procedure for $h^k_{\max}$ satisfies the following rules:
\begin{equation}
\label{eq:update:hmax}
    h^{k+1}_\text{max} \ = \ 
    \begin{cases}
        \displaystyle\max_{x\in \C^{k+1}} \{h(x) : 0<h(x)< h^k_\text{max} \}
        & \text{if the iteration is improving}, \\[2ex]
        h^{k+1}_\text{I}
        & \text{otherwise}.
    \end{cases}
\end{equation}

The parameter $h^k_\text{max}$ gradually decreases toward zero and determines the tolerance allowed for constraint violations. 
When the iteration is improving, the idea is to push the infeasible incumbent values toward the feasible region even when iteration $k$ does not generate a point that dominates the current incumbents. 
After a dominating or unsuccessful iteration, this parameter is set equal to $h^{k}_\text{I}$.
In an improving iteration, the existence of a visited point $t$ such that $0<h(t)<h(x^k_\text{I})$ implies that $U^k\setminus \{x^k_\text{I}\}$ is nonempty. 
It is therefore possible to reduce the threshold $h^k_\text{max}$ while keeping at least one admissible infeasible point. 
The chosen strategy is to set $h^k_\text{max}$ to the largest value of $h(x)$ that is strictly smaller than the previous threshold and such that $x$ is an improving point. Setting $h^k_\text{max}=0$ would reduce the problem to the minimization of the \textit{extreme barrier function} $f_\Omega$, where $f_\Omega=f$ on $\Omega$ and $f_\Omega=+\infty$ elsewhere.

The general PB framework 
is summarized in \cref{algo-pb}.
In the step 2, the set $\mathbb{Y}^k$ of points generated at iteration $k$ is arbitrary.

\begin{algorithm}[htb!]
\caption{The Progressive Barrier (PB) framework.}
\label{algo-pb}
\begingroup\small
\begin{algorithmic}[0]
\newcommand{\step}[1]{\Statex{\color{blue}\bfseries$\triangleright$ #1}}

\Statex\smallskip {\color{blue}\textbf{$\triangleright$ Step~0. Initialization:}}
\Statex
\hspace{0.2em}%
\begin{minipage}{\dimexpr\linewidth-1.5em\relax}
\begin{tcolorbox}[
  enhanced,
  colback=white,
  boxrule=0pt,
  frame hidden,
  borderline west={0.3pt}{0pt}{blue},
  sharp corners,
  left=0em,
  right=0pt,
  top=0pt,
  bottom=0pt,
  boxsep=0pt
]
\begin{algorithmic}

\Statex\smallskip  Initialize $\C^0\gets \{x^0\}$, set $h^0_{\max}\gets +\infty$ and $k\gets 0$
\end{algorithmic}
\end{tcolorbox}
\end{minipage}

\Statex\smallskip {\color{blue}\textbf{$\triangleright$ Step~1. Incumbents definitions:}}

 \Statex
\hspace{0.2em}%
\begin{minipage}{\dimexpr\linewidth-1.5em\relax}
\begin{tcolorbox}[
  enhanced,
  colback=white,
  boxrule=0pt,
  frame hidden,
  borderline west={0.3pt}{0pt}{blue},
  sharp corners,
  left=0em,
  right=0pt,
  top=0pt,
  bottom=0pt,
  boxsep=0pt
]
\begin{algorithmic}

\Statex\smallskip  
Set $x^k_{\text{F}}$ and $x^k_{\text{I}}$ according to \eqref{eq:xkF} and \eqref{eq:xkI}, respectively.

\end{algorithmic}
\end{tcolorbox}
\end{minipage}

\Statex\smallskip {\color{blue}\textbf{$\triangleright$ Step~2. Function evaluations:}}
 \Statex
\hspace{0.2em}%
\begin{minipage}{\dimexpr\linewidth-1.5em\relax}
\begin{tcolorbox}[
  enhanced,
  colback=white,
  boxrule=0pt,
  frame hidden,
  borderline west={0.3pt}{0pt}{blue},
  sharp corners,
  left=0em,
  right=0pt,
  top=0pt,
  bottom=0pt,
  boxsep=0pt
]
\begin{algorithmic}

\Statex\smallskip 
Generate $\mathbb{Y}^k$, evaluate $(f(y),h(y))$ for all $y\in\mathbb{Y}^k$, and set $\C^{k+1}\gets \C^k \cup \mathbb{Y}^k$

\end{algorithmic}
\end{tcolorbox}
\end{minipage}


\Statex\smallskip {\color{blue}\textbf{$\triangleright$ Step~3. Update step:}}
 \Statex
\hspace{0.2em}%
\begin{minipage}{\dimexpr\linewidth-1.5em\relax}
\begin{tcolorbox}[
  enhanced,
  colback=white,
  boxrule=0pt,
  frame hidden,
  borderline west={0.3pt}{0pt}{blue},
  sharp corners,
  left=0em,
  right=0pt,
  top=0pt,
  bottom=0pt,
  boxsep=0pt
]
\begin{algorithmic}

 \Statex   Declare iteration $k$ $\left\{ \begin{array}{ll}
        \textit{dominating} & \mbox{if there exists a dominating point } y\in\C^{k+1}\\
        \textit{improving} & \mbox{else, if there exists an improving point } y\in\C^{k+1}\\
        \textit{unsuccessful} & \mbox{otherwise} \\
    \end{array}\right.$

\Statex Update $h^{k+1}_{\max}$ according to~\eqref{eq:update:hmax}

\end{algorithmic}
\end{tcolorbox}
\end{minipage}

\Statex\smallskip {\color{blue}\textbf{$\triangleright$ Step~4. Termination:}}
 \Statex
\hspace{0.2em}%
\begin{minipage}{\dimexpr\linewidth-1.5em\relax}
\begin{tcolorbox}[
  enhanced,
  colback=white,
  boxrule=0pt,
  frame hidden,
  borderline west={0.3pt}{0pt}{blue},
  sharp corners,
  left=0em,
  right=0pt,
  top=0pt,
  bottom=0pt,
  boxsep=0pt
]
\begin{algorithmic}  \Statex \textbf{if} no termination test is triggered \textbf{then } $k\gets k+1$ and return to \color{blue}\textbf{Step~1}
\end{algorithmic}
\end{tcolorbox}
\end{minipage}

\end{algorithmic}
\endgroup
\end{algorithm}


\subsection{A motivational toy example}
While the PB provides a general framework to handle constraints, it does not, by itself, ensure convergence; such guaranties are inherited from the optimization algorithm in which it is embedded. 
With the Mesh Adaptive Direct Search (\mads) algorithm~\cite{AuDe09a}, these guaranties rely on the use of a mesh that discretizes the search space and governs the generation of trial points. 
This dependence on the mesh, although central to the convergence analysis, can also introduce limitations.
The \mads candidate points are generated on a discretization of the space, controlled by a mesh size parameter. When the algorithm proposes a point that is not on the current mesh, the point is projected onto a nearby mesh point before evaluation. 
This mechanism is central in \mads, but it can impede the benefit of accurate model-based information: even if a local model suggests a trial point close to a local minimizer $x^*$, the mesh projection may move it away, and the algorithm would need to undergo many mesh refinements before evaluating $f$ on a trial point close to $x^*$. 

To illustrate this last observation, consider the  smooth convex problem in $\R^2$
\begin{equation*}
    \min_{(x_1,x_2)\in\R^2}\ f(x_1,x_2)=x_1^2+x_2^2
\quad \text{s.t.}\quad
x_1-3x_2\le 0
\text{ and } \
x_2-3x_1\le 0,
\end{equation*}
whose unique optimal solution is $x^\star=\left(0,0\right)$. \Cref{fig:comparison_plots} illustrates the obtained results of two DFO algorithms using the starting point $(5/3,5/3)$.
In~\cref{fig:evaluations_comparison}, the feasible domain $\Omega$ is represented in green. 
After a small number of function evaluations, 
    a quadratic model can be built to locate exactly the optimal solution $x^\star$. 
With \madspb~\cite{AuDe09a} (i.e. \mads with the PB),
    all trial points need to be projected onto a mesh,
    which forces additional evaluations until the mesh becomes fine enough to locate $x^\star$ with sufficient precision. 
The left part of the figure shows the feasible incumbent solutions produced by \madspb and by a model-based algorithm (MBA), which does not project trial points onto a mesh.
In this setting, the MBA builds local quadratic models of the objective and constraint functions from previously visited points, and proposes trial points by solving the resulting surrogate problem. 
In this example, as shown in~\cref{fig:convergence_comparison}, \madspb~fails to evaluate $f$ on the exact optimal point, while the MBA reaches $x^\star$ in fewer than ten evaluations.

\begin{figure}[htb!]
\centering
\begin{subfigure}[t]{0.50\textwidth}
    \centering
    \includegraphics[width=\linewidth]{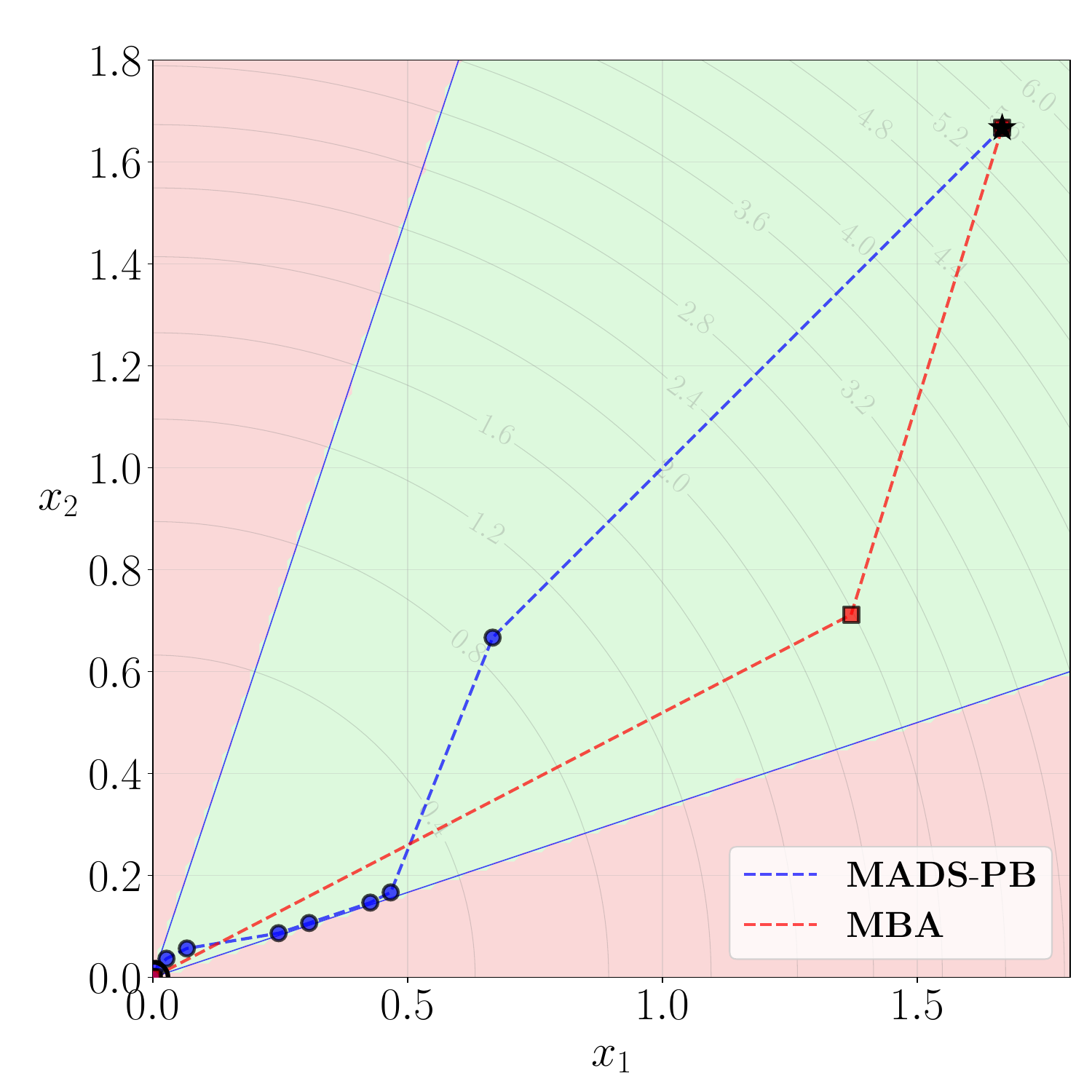}
    \caption{Trajectory plots.}
    \label{fig:evaluations_comparison}
\end{subfigure}
\hfill
\begin{subfigure}[t]{0.475\textwidth}
    \centering
    \includegraphics[width=\linewidth]{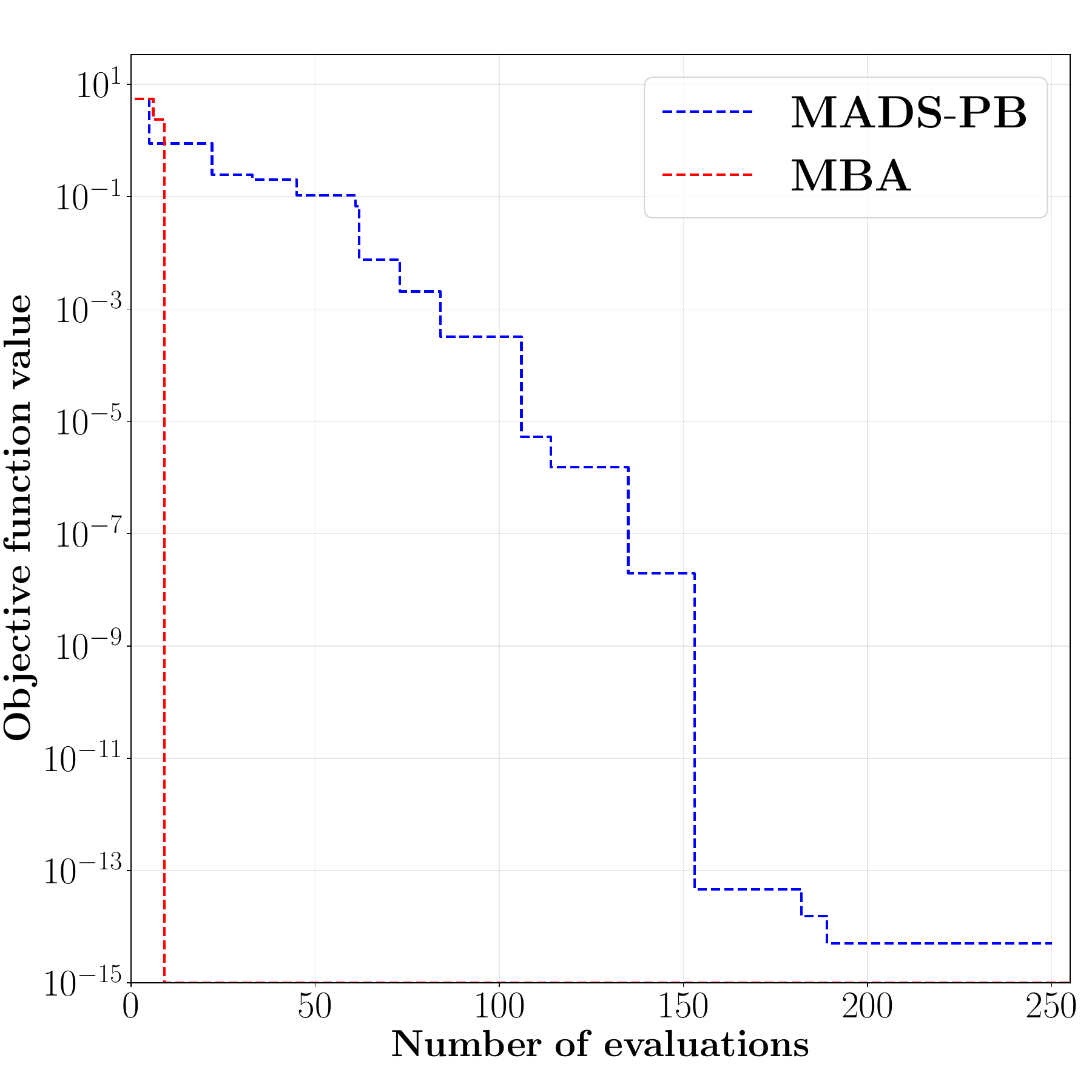}
    \caption{Convergence plots.}
    \label{fig:convergence_comparison}
\end{subfigure}
\caption{Trajectory and convergence plots for the MADS-PB method and a model-based algorithm (MBA).}
\label{fig:comparison_plots}
\end{figure}

\section{Adaptive direct search with a progressive barrier}
\label{sec:ADS}
This section describes the main contributions of this work: the \adspb algorithm,
based on the general PB framework described in \cref{algo-pb}.
The \adspb algorithm extends \ads~\cite{denorme-ads-2025} to constrained blackbox optimization problems with quantifiable and relaxable constraints by incorporating the PB mechanism of \madspb~\cite{AuDe09a}. 
The main purpose is to retain the PB framework while replacing the mesh-based structure of \madspb by the mesh-free structure of \ads.
In particular, the step~2 is decomposed into the \textit{Search} and \textit{Poll steps}.  
In addition, a fourth iteration status is introduced in the \textit{Update step}:
    iteration $k$ can be declared {\em reframing}. 
The overall structure of the method is summarized in \cref{algo-adspb}.

\subsection{The unrestricted search step}
\label{subsec-search}

The \emph{search} step at iteration $k$ of \adspb produces a finite, possibly empty set of trial points in $X$ denoted by $\mathbb{S}^k$.
The generation of these points is flexible and may exploit information from previously visited points:
let $\V^k_\text{succ}$ denote the set of successful visited points up to the start of iteration $k$. The first set $\V^0_\text{succ}$ is initialized to $\{x^0\}$ where $x^0\in X$.

Unlike \madspb, the search set $\mathbb{S}^k$ is not restricted to a mesh but to a punctured space consisting of $\R^n$  deprived of balls around all points of $\V^k_\text{succ}$.
The radii of the balls are determined by the {\em exclusion size parameter} $\delta^k>0$, 
 which is updated at the end of each iteration.
This parameter plays a role similar to that of the \mads mesh size parameter.

\begin{definition}\label{def:esh-NEW}
The {\em punctured space} $\esh$ at iteration $k$ of \adspb is the set of points in $\R^n$ that are not within $\deltA^k$ of $\V^k_\text{succ}$, i.e.,
\begin{equation}\label{eq:esh}
\esh \ := \ \left\{ x \in \R^n : \Vert x-y\Vert \geq \deltA^k \mbox{ for all } y \in \V^k_\text{succ} \right\}\ .
\end{equation}
\end{definition}

With \adspb, there are three possible outcomes of the search step:

\begin{enumerate}
    \item A dominating trial point $t_\text{search}\in\mathbb{S}^k$ inside $\esh$ is found: As soon as this occurs, in \textit{opportunistic} settings the search is interrupted, the iteration is declared \textit{dominating}, \adspb skips the poll step, and pursues to the update step. Note that during the poll step the notion of dominating uses $p^k_\text{F}$ and $p^k_\text{I}$ as reference instead of $x^k_\text{F}$ and $x^k_\text{I}$.
    \item A dominating trial point $t_\text{search}\in\mathbb{S}^k$ outside $\esh$ is found: As soon as this occurs, the search is interrupted and \adspb performs the poll step around $t_\text{search}$ (described in the next subsection).
    The iteration will be declared \textit{dominating} if the poll step identifies a dominating point, and \textit{reframing} if the poll step fails to do so.
    \item All trial points of $\mathbb{S}^k$ are evaluated, and none of them is dominating.\\
    \adspb performs the poll step around the incumbent(s) solution(s) (described in the next subsection).
    The poll step will determine if the iteration is declared \textit{dominating}, \textit{improving} or \textit{unsuccessful}.
\end{enumerate}

The set of all points on which $f$ is evaluated by the algorithm at the search step of iteration $k$ is denoted $\mathbb{S}^k_\text{eval}$.
This set might be a strict subset of $\mathbb{S}^k$ when the search step is opportunistically interrupted.

\subsection{The poll step confined to the punctured space}
\label{subsec-poll}
The \emph{poll} step is a local exploration around the best feasible and infeasible solutions found so far.  
The algorithm defines the \textit{poll centers} as follows:

\begin{definition}[Poll centers]
\label{def:poll:centers}
The feasible and infeasible poll centers are defined as
\begin{equation*}
p_{\textnormal{F}}^k=
\begin{cases}
t_{\textnormal{search}} & \text{if } t_{\textnormal{search}}\in\Omega \text{ and dominating},\\
x_{\textnormal{F}}^k, & \text{otherwise if defined.}
\end{cases}
\
p_{\text{I}}^k=
\begin{cases}
t_{\textnormal{search}}, & \text{if } t_{\textnormal{search}}\notin\Omega \text{ and dominating},\\
x_{\textnormal{I}}^k, & \text{otherwise if defined.}
\end{cases} 
\end{equation*}
\end{definition} 
In other words, at the beginning of the poll step of iteration $k$, $p_{\text{F}}^k$ denotes the best feasible point identified so far, while $p_{\text{I}}^k$ denotes the best infeasible point identified so far among those satisfying $h(x)\leq h_{\max}^k$.
Moreover, there are either one or two poll centers since $F^k\cup I^k$ is nonempty (since $x^0 \in F^0 \cup I^0$).

Once that the poll centers are determined, the poll set $\mathbb{P}^k$ is built from the sets $\mathbb{D}^k_\text{I}$ and $\mathbb{D}^k_\text{F}$ of normalized \textit{poll directions}.
These directions, i.e., unit vectors defining the directions along which trial points are generated, are  
selected to form a positive spanning set~\cite{Davi54b}, ensuring adequate coverage of the local region around $p^k_\text{F}$ and $p^k_\text{I}$. 
The \textit{tentative poll points} are then obtained by moving from the poll centers along these directions, scaled by the \textit{frame size parameter} $\Delta^k \in \R_+^*$, introduced in~\cite{AuDe2006}, since this parameter plays the same role as in \madspb.
The poll set is defined as

$$ \P^k \ := \left\{ \renewcommand{\arraystretch}{1.2}
\begin{array}{cccll}
         \{p_\text{F}^k + \Delta^k v : v \in \mathbb{D}_\text{F}^k\} &&&&\text{ if } \ I^k =\emptyset, \\
         \{p_\text{I}^k + \Delta^k v  : v \in \mathbb{D}_\text{I}^k\}&& &&\text{ if } \ F^k =\emptyset, \\
         \{p_\text{F}^k + \Delta^k v : v \in \mathbb{D}_\text{F}^k\} &\cup& \{p_\text{I}^k + \Delta^k v : v \in \mathbb{D}_\text{I}^k\} &&\text{ otherwise. } 
     \end{array} \right. 
$$

As with \ads, the parameter $\Delta^k$ determines the distance between the trial points in $\P^k$ and the poll centers $p_\text{I}^k$ and $p_\text{F}^k$.  
Because the poll directions are normalized, every poll point lies exactly at distance $\Delta^k$ from its associated poll center.  
When $\Delta^k$ decreases, the tentative poll points move closer to $p_\text{I}^k$ and $p_\text{F}^k$, allowing finer local improvements.  
All tentative poll points in $\P^k$ that fall outside the punctured space $\esh$ are discarded without evaluation, as such points are within distance $\delta^k$ of some element of $\V_\text{succ}^k$.  

There are two possible outcomes of the poll step:
\begin{enumerate}
    \item The poll step successfully produces a trial point  $t_\text{poll} \in \esh \cap \P^k$ that is dominating. In an opportunistic setting, the poll step is immediately interrupted.
    \item After evaluating all points of $ \esh \cap \P^k$, no dominating trial point is found. 
\end{enumerate}
In the situation where the search step of iteration $k$ visits a feasible or infeasible dominating point $t_\text{search}\in \mathbb{S}^k_\text{eval}$ outside of $\esh$, then either $p^k_\text{F} = t_\text{search}$ or $p^k_\text{I} = t_\text{search}$. 
Two situations may occur. If the subsequent poll step fails to identify a dominating point, then the iteration is declared \textit{reframing} and $t_\text{search}$ becomes one of the next incumbents. 
Alternately, if the poll step produces a dominating point $t_\text{poll}$, the iteration is declared \textit{dominating} and one incumbent is updated to $t_\text{poll}$. 
In both cases, at least one incumbent is updated at iteration~$k$.

Denote by $\P^k_\text{eval}$ the set of tentative poll points at which $f$ is evaluated during the poll step.  
Once this step is completed, an update step is performed.

\subsection{Decision and parameters update}
\label{subsec-update}
The \textit{update} step is the final stage of iteration $k$. This step consists in analyzing the visited points during iteration $k$ and the points from $\mathbb{C}^k$ in order to update all the parameters of \adspb. 
When either the search or the poll step produced a dominating point in $\esh$, then iteration $k$ is declared \emph{dominating}. Otherwise, before declaring the iteration unsuccessful or reframing, the algorithm checks whether iteration $k$ can be declared \emph{improving}.

At iteration $k$, $\mathbb{C}^k \cup \mathbb{S}^k_\text{eval} \cup \mathbb{P}^k_\text{eval}$
denotes the set gathering all points visited before iteration $k$ together with the trial points visited during the search and poll steps of iteration $k$.
Iteration $k$ is said to be \textit{improving} if there exists a non dominating point $t_\text{update} \in \C^k \cup \mathbb{S}^k_\text{eval} \cup \mathbb{P}^k_\text{eval}$ such that
\begin{equation}
     t_\text{update}\in \esh \ \text{ and }\  0 < h(t_\text{update}) < h\!\left(x^k_\text{I}\right).
\end{equation} 
By construction, the point $t_\text{update}$ must be an infeasible point that is strictly better than the current infeasible incumbent $x^k_\text{I}$ in terms of constraint violation. Such a point is called an \textit{improving point}. 

Finally, if an iteration is neither dominating nor improving, then it is declared \textit{reframing} whenever the search step has produced a dominating point outside $\esh$, and \textit{unsuccessful} otherwise.
This terminology (i.e., a \textit{reframing} iteration) is motivated by the fact that, even if no dominating point is found within the punctured space, a dominating search point outside $\esh$ changes the incumbents. 
As a result, the poll centers as well as the corresponding frames are repositioned. 
The update procedure for the barrier threshold parameter $h^k_{\max}$ follows the PB rule of~\cite{AuDe09a} and is recalled in \cref{sec:PB}, see~\eqref{eq:update:hmax}.
In particular, the threshold is aligned with the current infeasible incumbent after a dominating, reframing or unsuccessful iteration, while an improving iteration decreases the threshold to the largest violation level strictly below the previous one among the infeasible visited points.

The update procedure for the frame size and exclusion parameters is defined as follows:
\begin{equation}
\label{def:update:deltas}
    (\DeltA^{k+1},\deltA^{k+1}) =
    \begin{cases}
        \increaset(\DeltA^{k},\deltA^{k})
        & \text{if the iteration $k$ is dominating}, \\[1ex]
        (\DeltA^{k},\deltA^{k})
        & \text{if the iteration $k$ is improving}, \\[1ex]
        \decreaset(\DeltA^{k},\deltA^{k})
        & \text{otherwise.} \\
    \end{cases}
\end{equation}
The $\decreaset$ and $\increaset$ rules are defined as in~\cite[Equations~(2) and~(3)]{denorme-ads-2025}. They allow that if $\lim_{k\in \N} \deltA^k=0$, then $\lim_{k\in \N} \dfrac{\deltA^k}{\DeltA^k}=0$ and $\lim_{k\in \N} \DeltA^k=0$.
The set $\V^k_\text{succ}$ is then updated as follows. 
\begin{equation}
\label{updateVk}
    \V^{k+1}_\text{succ} = \begin{cases}
        \V^{k}_\text{succ} \cup \{t_\text{search}\} \ &\text{ if } t_\text{search}\in \mathbb{S}^k_\text{eval}\cap \esh  \text{ is dominating} \\
         \V^{k}_\text{succ} \cup \{t_\text{poll}\} \ &\text{ if } t_\text{poll}\in \mathbb{P}^k_\text{eval} \text{ is dominating,} \\
          \V^{k}_\text{succ} \cup \{t_\text{update}\}  &\text{ if } t_\text{update}\in\mathbb{C}^k \cup \mathbb{S}^k_\text{eval} \cup \mathbb{P}^k_\text{eval} \cap \esh  \text{ is improving,} \\
          \V^k_\text{succ} & \text{ otherwise.} 
    \end{cases}
\end{equation}

Note that the ``otherwise'' cases in~\eqref{def:update:deltas} and~\eqref{updateVk} actually encompass two distinct situations, namely reframing  and unsuccessful iterations. 
Although both cases lead to the same updates of the parameters $(\DeltA^k,\deltA^k)$ and the set $\V^k_{\mathrm{succ}}$, the outcome of the iteration is fundamentally different. In the case of a reframing iteration, an incumbent solution is guaranteed to be modified, reflecting a meaningful improvement. 
In contrast, during an unsuccessful iteration, no incumbent is updated and no progress is made in terms of dominance or improvement. 

In \cref{algo-adspb}, the temporary set $T^k_{\text{succ}}$ is introduced to ensure compact notation.
The set of successfully visited points $\V^k_{\mathrm{succ}}$ replaces the set $\V^k$ introduced in~\cite{denorme-ads-2025} and is used to construct the punctured space $\esh$. At the end of each iteration, the set $\V^k_{\mathrm{succ}}$ may be updated by adding at most one new point. A point is added to $\V^k_{\mathrm{succ}}$ only when iteration $k$ is declared improving or dominating. Moreover, whenever iteration $k$ is improving or dominating, the point added to $\V^k_{\mathrm{succ}}$ is guaranteed to be at least $\deltA^k$ units away from all previously stored points in $\V^k_{\mathrm{succ}}$.

This preserves the foundational principles of directional direct search while removing the limitations imposed by mesh structures in \madspb.
The set of successfully visited points $\V^k_{\mathrm{succ}}$ replaces the set $\V^k$ introduced in~\cite{denorme-ads-2025} and is used to construct the punctured space $\esh$. At the end of each iteration, the set $\V^k_{\mathrm{succ}}$ may be updated by adding up to one new point. A point is added to $\V^k_{\mathrm{succ}}$ only when iteration $k$ is declared improving or dominating. 

There is three main differences between \adspb and \madspb. 
First, the search step of \adspb is completely free in $X$: any trial point in the blackbox domain may be evaluated, without projection onto a mesh. 
Second, before evaluating poll points, \adspb checks whether they belong to the punctured space $\esh$, so as to avoid evaluating points that are too close to previously successful points. 
Third, this mechanism introduces a new possible qualification of the iteration, called a reframing iteration, which occurs when the search step finds a successful point outside the punctured space and then the poll fails to find a dominating point.

\begin{algorithm}[htb!]
\caption{Adaptive Direct Search with Progressive Barrier (\adspb)}
\label{algo-adspb}
\begin{algorithmic}[0]

\Statex\smallskip {\color{blue}\textbf{$\triangleright$ Step~0. Initialization:}}
\Statex
\hspace{0.2em}%
\begin{minipage}{\dimexpr\linewidth-1.5em\relax}
\begin{tcolorbox}[
  enhanced,
  colback=white,
  boxrule=0pt,
  frame hidden,
  borderline west={0.3pt}{0pt}{blue},
  sharp corners,
  left=0em,
  right=0pt,
  top=0pt,
  bottom=0pt,
  boxsep=0pt
]
\begin{algorithmic}
\Statex $x^0 \in X$, $\DeltA^0 \in \R_+^*$, $\deltA^0 \in \R_+^*$ such that $\deltA^0 \le \DeltA^0$ \\
$h^0_{\max} \gets +\infty$, $k \gets 0$ \hfill {\color{blue}$\triangleright$ initial parameters} \\
$\C^0 \gets \{x^0\}$ \hfill {\color{blue}$\triangleright$ initial set of cache points} \\
$\V^0_{\text{succ}} \gets \{x^0\}$ \hfill {\color{blue}$\triangleright$ initial set of successful points} \\
$\eshzero \gets \{x \in \R^n : \|x-x^0\| \ge \deltA^0\}$ \hfill {\color{blue}$\triangleright$ initial punctured space}

\end{algorithmic}
\end{tcolorbox}
\end{minipage}

\Statex\smallskip {\color{blue}\textbf{$\triangleright$ Step~1. Incumbents definition:}}
\Statex
\hspace{0.2em}%
\begin{minipage}{\dimexpr\linewidth-1.5em\relax}
\begin{tcolorbox}[
  enhanced,
  colback=white,
  boxrule=0pt,
  frame hidden,
  borderline west={0.3pt}{0pt}{blue},
  sharp corners,
  left=0em,
  right=0pt,
  top=0pt,
  bottom=0pt,
  boxsep=0pt
]
\begin{algorithmic}

\Statex\smallskip Set the incumbents $x^k_{\text{F}}$ and $x^k_{\text{I}}$
(see \cref{sub:incumbent}) and
$T^k_{\text{succ}} \gets \emptyset$

\end{algorithmic}
\end{tcolorbox}
\end{minipage}

\Statex\smallskip {\color{blue}\textbf{$\triangleright$ Step~2. Function evaluations:}}

\Statex
\hspace{0.2em}%
\begin{minipage}{\dimexpr\linewidth-1.5em\relax}
\begin{tcolorbox}[
  enhanced,
  colback=white,
  boxrule=0pt,
  frame hidden,
  borderline west={0.3pt}{0pt}{blue},
  sharp corners,
  left=0em,
  right=0pt,
  top=0pt,
  bottom=0pt,
  boxsep=0pt
]
\begin{algorithmic}

\Statex\smallskip {\color{blue}\textbf{$\triangleright$ Step~2.1. Search step:}}
\State Define a finite search set $\mathbb{S}^k \subset X$
\If{$t_{\text{search}}$ is \textit{dominating} for some $t_{\text{search}} \in \mathbb{S}^k_{\text{eval}}$}
  \If{$t_{\text{search}} \in \esh$}
    \State Set  $T^k_{\text{succ}} \gets \{t_{\text{search}}\}$ and go to {\color{blue}\textbf{Step~3}}
  \EndIf
\EndIf

\Statex\smallskip {\color{blue}\textbf{$\triangleright$ Step~2.2. Poll step:}}
\State Define the poll set $\mathbb{P}^k=\mathbb{P}^k_\text{F}\cup \mathbb{P}^k_\text{I}$ around $p_{\text{F}}^k$ and $p_{\text{I}}^k$ (see \cref{def:poll:centers})
\If{$t_{\text{poll}}$ is \textit{dominating} for some
$t_{\text{poll}} \in \mathbb{P}^k_{\text{eval}} \cap \esh$}
  \State Set $T^k_{\text{succ}} \gets \{t_{\text{poll}}\}$ 
\EndIf

\end{algorithmic}
\end{tcolorbox}
\end{minipage}

\Statex\smallskip {\color{blue}\textbf{$\triangleright$ Step~3. Update step:}}
\Statex
\hspace{0.2em}%
\begin{minipage}{\dimexpr\linewidth-1.5em\relax}
\begin{tcolorbox}[
  enhanced,
  colback=white,
  boxrule=0pt,
  frame hidden,
  borderline west={0.3pt}{0pt}{blue},
  sharp corners,
  left=0em,
  right=0pt,
  top=0pt,
  bottom=0pt,
  boxsep=0pt
]
\begin{algorithmic}
\State Set $\C^{k+1} \gets \C^k \cup \mathbb{S}^k_{\text{eval}} \cup \mathbb{P}^k_{\text{eval}}$
\Statex

\shortstack[l]{\begin{tabular}{l} Declare \\iteration $k$\end{tabular}}
$\left\{
\begin{array}{ll}
    \textit{dominating} & \mbox{\textbf{if }} T^k_{\text{succ}} \neq \emptyset\\[0.3em]
    \textit{reframing} & \mbox{\textbf{else, if }} \text{the search step produced a dominating point}\\[0.3em]
    \textit{improving} & \mbox{\textbf{else, if }}  \text{there is }t_\text{update}\in\C^{k+1}\cap \esh \text{ an improving point}\\[0.3em]
    \textit{unsuccessful} & \mbox{\textbf{otherwise}}
\end{array}
\right.$
\Statex \textbf{if } iteration $k$ is \textit{improving} \textbf{then} set $T^k_{\text{succ}} \gets \{t_{\text{update}}\}$
\Statex Update $\V^{k+1}_{\text{succ}} \gets \V^k_{\text{succ}} \cup T^k_{\text{succ}}$, and $\eshkplusun$ according to~\eqref{eq:esh}
\Statex Update $h^{k+1}_{\max}$ and $(\DeltA^{k+1}, \deltA^{k+1})$
according to~\eqref{eq:update:hmax} and~\eqref{def:update:deltas}

\end{algorithmic}
\end{tcolorbox}
\end{minipage}

\Statex\smallskip {\color{blue}\textbf{$\triangleright$ Step~4. Termination:}}
\Statex
\hspace{0.2em}%
\begin{minipage}{\dimexpr\linewidth-1.5em\relax}
\begin{tcolorbox}[
  enhanced,
  colback=white,
  boxrule=0pt,
  frame hidden,
  borderline west={0.3pt}{0pt}{blue},
  sharp corners,
  left=0em,
  right=0pt,
  top=0pt,
  bottom=0pt,
  boxsep=0pt
]
\begin{algorithmic}
\If{no termination test is triggered}
   $k \gets k+1$ and go to {\color{blue}\textbf{Step~1}}
\EndIf
\end{algorithmic}
\end{tcolorbox}
\end{minipage}
\end{algorithmic}
\end{algorithm}


\subsection{Convergence analysis}
\label{sec:convergence}

This section presents the convergence analysis of \adspb. 
As is common in direct search~\cite{AuDe03a},
 the analysis first studies the behavior of $\delta^k$ and $\Delta^k$ as $k$ goes to infinity,
and then proposes stationary results.

\begin{assumption} \label{asm:1}
All the points at which the objective and constraint functions are evaluated belong to a compact set $L$.
\end{assumption}

Under \cref{asm:1}, the exclusion radius parameter $\deltA^k$ and the frame size parameter $\DeltA^k$ converge to zero as the number of iterations approaches infinity.

\begin{theorem}
    \label{Th-limdelta}
Let \cref{asm:1} hold. 
The exclusion radius parameter sequence $\{\deltA^k\}_{k \in \N}$ and the frame size sequence $\{\DeltA^k\}_{k \in \N}$ produced by an instance of \adspb satisfy 
$$\displaystyle\lim_{k \rightarrow \infty} \deltA^k \ = \ 0 \qquad \text{ and } \qquad \displaystyle\lim_{k \rightarrow \infty} \DeltA^k \ =\  0.$$
\end{theorem}
\begin{proof}
Let $\varepsilon \in\R_+^*$ and define $\mathcal{K}_\varepsilon:=\left\{k\in\N:\deltA^k\geq \varepsilon\right\}$,
a subset of iteration indices. 

Suppose that $\mathcal{K}_\varepsilon$ contains infinitely many elements.
Let $\mathcal{S}_{\varepsilon}\subseteq \mathcal{K}_\varepsilon$ denote the subset of indices of dominating and improving iterations $k$ for which $\deltA^k\geq \varepsilon$.
Two scenarios can occur depending on the cardinality of $\mathcal{S}_{\varepsilon}$: 
\begin{itemize}
    \item
$|\mathcal{S}_{\varepsilon}|$ is finite, in which case there exists $k_0 \in \N$ such that the iteration $k$ is unsuccessful or reframing for all $k\geq k_0$.
It follows that $ \lim_{k\in\mathcal{K}_\varepsilon }\deltA^k=0$,
 contradicting $\deltA^k\geq \varepsilon$, for all $k\in\mathcal{K}_\varepsilon$.
    \item  
$|\mathcal{S}_\varepsilon|$ is infinite.
For any $k \in \mathcal{S}_\varepsilon$, 
    let $t^k \in \{ t_{\text{search}}, t_{\text{poll}}, t_{\text{update}}\}$ denote the improving or dominating point found during the iteration.  
    Since the iteration is either dominating or improving,
    it follows that $t^k$ belongs to $\esh$ and 
$$ t^k \ \in \ \left\{ t \in X \,:\, \Vert t-t^j\Vert \geq \delta^k,~j \in \mathcal{S}_\varepsilon \text{ and } j< k\right\}.$$
Hence, any pair of indices $k_1 \ne k_2$ in $\mathcal{S}_\varepsilon$
 satisfy the inequality
\begin{equation}
    \label{thlimeq}
    \|t^{k_1} -t^{k_2} \| \ \ge \ \inf_{i\in \mathcal{S}_\varepsilon} \deltA^i  \ \ge \ \varepsilon\,.
\end{equation} 
However, since $\{t^k\}_{k \in \mathcal{S}_\varepsilon}$ belongs to the compact set $\compact$ (from \cref{asm:1}), the Bolzano-Weierstrass theorem ensures the existence of a convergent subsequence, contradicting~\eqref{thlimeq}. 
\end{itemize}
Both scenarios lead to a contradiction, and therefore the set $\mathcal{K}_\varepsilon$ is finite.  
Finally, the discussion following~\eqref{def:update:deltas}
ensures that $\{\deltA^k\}_{k\in\N} \to 0$ implies that $\{\DeltA^k\}_{k\in\N} \to 0$.
\end{proof}


\Cref{Th-limdelta} is stronger than the corresponding result in the \madspb context~\cite[Proposition~3.1]{AuDe2006}, as the limit is shown to be zero, instead of simply ensuring that the limit inferior is zero.
This stronger limit result is similar to those derived in sufficient decrease contexts~\cite[Theorem~4.1]{diouane2022trego} and~\cite[Lemma~3.2]{BeSoVi2023}.

The \adspb algorithm generates a sequence of trial points from which one can extract one or two subsequences of poll centers: the feasible poll centers $\{p^k_\text{F}\}_{k\in \N}$ and the infeasible poll centers $\{p^k_\text{I}\}_{k\in \N}$. 
The following convergence result focuses on the subsequence of feasible poll centers.

\begin{definition}[A refining subsequence]
\label{def:refining}
\cite[Definition~3.1]{AuDe2006} A convergent subsequence $\{p^k\}_{k \in \K}$,
where $\K$ is an infinite subset of $\N$, of unsuccessful iterations is said to be a refining subsequence if $\lim_{k \in \K} \Delta^k = 0 $. 
The limit $\hat x$ of a convergent refining subsequence is called a refined point. 
If $ \hat v = \lim_{k \in \L} v^k $ exists for some $\L \subseteq \K$ with poll direction $v^k \in \mathbb{P}^k$ and if $p^k+\DeltA^k v^k \in X$ for infinitely many $k \in \L$, then $\hat v$  is said to be a refining
direction for $\hat x$.
\end{definition}

The notion of a refining subsequence makes it possible to characterize the asymptotic behavior of the sequence of feasible incumbents along unsuccessful iterations as the frame size parameter tends to zero.
For the next convergence results, one must guaranty that
$p+t \hat v$ belongs to $\Omega$ when $p \in \Omega$ is close to $\hat x$, which is ensured if $\hat v$ is a hypertangent vector to $\Omega$ at $\hat x$. 

\begin{definition}{\cite{Rock80a}}
A vector $\hat v\in \R^n$ is said to be a hypertangent vector to the set $\Omega \subseteq \R^n$
at the point $\hat{x}\in\Omega$ if there exists a scalar $\epsilon>0$ such that
\begin{equation*}
p+tv\in\Omega \qquad\forall p\in\Omega\cap B_\epsilon(\hat{x}),\  v\in B_\epsilon(\hat v), \ \text{ and }\  0<t<\epsilon.
\end{equation*}
\end{definition}

The hypertangent cone to $\Omega$ at $\hat{x}$, denoted by $T_\Omega^H(\hat{x})$, is the set of all hypertangent vectors
to $\Omega$ at $\hat{x}$. In this case, for $\hat{x}\in \Omega$ and $\v \in \T(\hat x)$, if $f$ is Lipschitz continuous near $\hat x$, the Clarke Jahn generalized derivative, as defined in~\cite{Clar83a} and extended to constrained settings in~\cite{Jahn94a}, is given by
\begin{equation*}
  f^\circ(\hat x; \v) := \limsup_{\substack{p \to \hat x, ~ t \searrow 0, \\ p + t\v \in \Omega, ~ p\in\Omega}} \frac{f(p + t\v) - f(p)}{t}.
\end{equation*}
The next theorem establishes that any feasible refined point is Clarke stationary with respect to every refining direction in $\T(\hat x_\text{F})$.

\begin{theorem}
\label{th:fclark}
    Let \cref{asm:1} hold. 
    If $\{p^k_\textnormal{F}\}_{k\in \K}$ is a refining subsequence converging to $\hat x_\textnormal{F} \in \Omega$, near which $f$ is Lipschitz continuous, then for all refining directions $\hat v\in \T(\hat{x}_\textnormal{F})$,  
    \begin{equation*}
            f^\circ(\hat x_\textnormal{F}; \hat v)\geq 0.
    \end{equation*}
\end{theorem}
\begin{proof}
Let $\{p^k_\text{F}\}_{k\in \K}$, be a refining subsequence composed of feasible poll centers,
 with refined point $\hat x_\text{F} \in \Omega$, and
 suppose that $\hat v = \lim_{k \in \L} v^k $ is a refining direction. 

The definition of a refining sequence ensures that 
    $ \lim_{k \in \L}\Delta^k = 0$ 
    and  $p^{k}_\text{F} = \hat x_\text{F}$ for all $k \in \L$.

Now, consider $t^k_\text{F} = p^k_\text{F}+\DeltA^k v^k$, 
    the poll point at iteration $k \in \L$ associated to the direction $v^k$.
Two cases need to be analyzed.
First, if $t^k_\text{F}$ does not belongs to the punctured space $\esh$, 
 then the trial point $t^k_\text{F}$ is rejected as there exists a previously visited point $y^k \in \C^k \subseteq \C^{k+1}$, such that $ \|p^k_\text{F}+ \DeltA^kv^k-y^k\|<\deltA^k.$
The direction  
    $w_k = \frac1{\delta^k}(y^k-t^k)$
    satisfies $\|w^k\| \leq 1$
    and
\begin{equation}
    y^k \ = \ t^k_\text{F} + \delta^k w^k 
        \ = \ p^k_\text{F} + \Delta^k v^k + \delta^k w^k 
        \ = \ p^k_\text{F} + \Delta^k\left( v^k + \tfrac{\delta^k}{\Delta^k} w^k\right) 
        \ \in \ \C^{k+1}\,.
    \label{eq-y-esh}
\end{equation}
Second, if $t^k_\text{F}$ belongs to the punctured space $\esh$, then define $y^k = t^k_\text{F}$ and set $w^k=0$ so that~\eqref{eq-y-esh} remains valid.
In both cases, $y^k \in \C^{k+1}$ and $f(y^k) \geq f(p^k_\text{F})$.

Now, \cref{Th-limdelta} ensures that 
    $\lim_{k \in \L}\deltA^k = \lim_{k \in \L}\DeltA^k = 0$, and so $ \lim_{k\in \L} \dfrac{\deltA^k}{\DeltA^k}=0$
which implies that
\begin{equation*}
    \displaystyle \lim_{k\in \L} v^k+\dfrac{\deltA^k}{\DeltA^k}w^k \ =\ \hat v\, .
\end{equation*}
Furthermore, since $\hat v$ is an hypertangent direction, there exists an infinite subset $\W\subseteq \L$, such that for all $k\in \W$, 
$$y^k, \ p^k_\text{F} \ \in\  \Omega \quad \mbox{ and } \quad f(y^k)\ \geq \  f(p^k_\text{F}).$$
Since $f$ is Lipschitz continuous near $\hat x_\text{F}$,
 the Clarke derivative exists.
To derive a lower bound, consider the following three sequences indexed by $k \in \W$
    $$ p^k_\text{F} \to \hat x_\text{F}, \quad 
    \quad \DeltA^k \searrow 0 \quad 
    \mbox{ and } \quad 
    v^k +\dfrac{\deltA^k}{\DeltA^k}w^k \to \hat v$$
    and by~\cite[Proposition~3.9]{AuDe2006}, it follows that

\begin{align*}
  f^\circ( \hat x_\text{F}; \v) 
        \ = \limsup_{\substack{x \to \hat x_\text{F}, ~ t \searrow 0, \\ x + t\v \in \Omega, ~x\in\Omega}} \frac{f(x + t\v) - f(x)}{t} 
        &= \limsup_{\substack{x \to \hat x_\text{F}, ~ t \searrow 0, \\ x + tv \in \Omega, ~x\in\Omega \\  v \to \v}} \dfrac{f(x+tv)-f(x)}{t}\\
        &\geq \ \limsup_{k \in \W} \frac{f\left(p^k_\text{F} + \DeltA^k  \left(v^k +\dfrac{\deltA^k}{\DeltA^k}w^k\right) \right) - f(p^k_\text{F})}{\DeltA^k}\\
        & =  \  \limsup_{k \in \W} \frac{f(y^k ) - f(p^k_\text{F})}{\DeltA^k} \ \geq \ 0.
\end{align*}
\end{proof}


The next result focuses on the sequence of infeasible poll centers $\{p^k_\text{I}\}_{k\in \N}$. 
There are three possible outcomes for this sequence: the first is that no infeasible point is ever found, in which case the sequence is not defined; 
the second is that there exists a refining infeasible subsequence converging to a refined point $\hat x_\text{I}$ in the feasible region $\Omega$; in this case, \adspb found a global minimizer of $h$; the last case is that there exists a refining infeasible subsequence converging to a refined point outside the feasible region $\Omega$; in this case, a necessary condition for $\hat x_\text{I}$ to be a local minimizer of $h$ is given in the following theorem.

\begin{theorem}
    Let \cref{asm:1} hold. If $\{p^k_\text{I}\}_{k\in \K}$ is a refining subsequence of infeasible poll centers converging to $\hat x_\text{I} \in X$ near which $h$ is Lipschitz with $x^k_\text{I} \in I^k$, then for all refining directions $\hat v\in T^\text{H}_{X}(\hat{x}_\text{I})$ 
    \begin{equation*}
            h^\circ(\hat x_\text{I}; \hat v)\geq 0.
    \end{equation*}
\end{theorem}
\begin{proof}
   Indeed, if $\hat x_\text{I} \in \Omega$, then $h(\hat x_\text{I})=0$. Since $h(x)\ge 0$ for all $x\in X$, it follows that $\hat x_\text{I}$ is a global minimizer of $h$ on $X$, and the result holds. Otherwise, if $\hat x_\text{I} \in X\setminus \Omega$, then by continuity of $h$, there exists $\varepsilon>0$ such that $h(x)>0$ for all $x\in B_\varepsilon(\hat x_\text{I})$. This ensures the existence of an infinite subset $\L\subseteq \K$ such that, for all $k\in \L$, the sequence $y^k$ defined in~\eqref{eq-y-esh} satisfies $y^k, p^k_\text{I} \in X\setminus \Omega$ and $h(y^k)\ge h(p^k_\text{I})$.

The conclusion follows by applying the same reasoning as in the proof of \cref{th:fclark}, replacing $f$ by $h$, $F^k$ by $I^k$, and $p^k_{\textnormal{F}}$ by $p^k_{\textnormal{I}}$.
\end{proof}

\section{Computational results}
\label{sec:numerical}

This section evaluates the practical performance of \adspb on a collection of constrained optimization problems. 
The proposed framework is tested on both analytical test problems and constrained blackbox problems with quantifiable and relaxable constraints. 
The objective is to empirically assess the behavior of \adspb and to compare it with existing DFO methods.

The experiments are conducted using the \nomad software~\cite{nomad4paper}, which provides a rich algorithmic environment for DFO. 
In particular, the computational tests reported in this section benefit from several search strategies available in the solver, including a quadratic model-based search~\cite{CoLed2011}, a Nelder-Mead search~\cite{AuTr2018}, and a speculative search strategy~\cite{AuDe2006}. 
The proposed approach is evaluated against other algorithms implemented within \nomad and against
other DFO solvers.
These include the Covariance Matrix Adaptation Evolution Strategy (CMA-ES)~\cite{cmaes}, the Non-dominated Sorting Genetic Algorithm II (NSGA-II)~\cite{Deb2000} and the mixed penalty-interior point method and direct search (LOGDS)~\cite{BrCuLiSi2024}.

The computational tests consider two classes of problems. The first class consists of analytical benchmark problems drawn from the literature, for which the objective and constraint functions are explicitly defined. The second class consists of constrained blackbox optimization problems, where function evaluations may be expensive and no analytical expressions are available. On both classes, \adspb is compared with the aforementioned methods in order to assess its performance in terms of robustness and efficiency. Comparison are done using data profiles~\cite{G-2025-36,MoWi2009}.
The results on analytical problems are presented in the next subsections.

\subsection[Implementation in NOMAD]{Implementation in \nomad}

The \adspb framework is implemented within the \nomad software as a specialized variant of the PB version of \madspb.
This implementation was carried out in a development version of \nomad and reuses part of the existing PB infrastructure while modifying several core components of the iteration logic.

More precisely, \adspb differs from the standard implementation of \madspb in the following aspects. First, \adspb does not rely on any mesh structure, and therefore no mesh projection is performed before evaluating $f$ on trial points. Second, the set of successful points, denoted by $\V^k_{\text{succ}}$, is explicitly maintained throughout the optimization process. 
This set is used to define the punctured space $\esh$.
Third, unlike in the standard \madspb implementation, the poll centers in \adspb are computed only after the search step has been completed. 
This design is consistent with the logic of the algorithm, where the outcome of the search step may affect the incumbent solutions and therefore the points around which polling should be performed.
Finally, poll points that do not belong to $\esh$ are excluded from evaluation. 
This ensures that the poll step only generates points satisfying the exclusion criterion, in accordance with the theoretical framework introduced in this work.

\subsection{Results on analytical problems from literature}

The set of test problems considered in this section is taken from~\cite{AuBrDiLeDSiTr26} as summarized in \cref{tab:problems}.
It consists of fourteen constrained problems. The number of variables $n$ ranges from $2$ to $13$, while the number of inequality constraints $m$ varies from $1$ to $15$. 
For each problem, several initial points are considered, including both feasible and infeasible ones. A problem instance is defined by a given initial point together with a random seed. Each instance is solved multiple times using $10$ different seeds in order to account for the stochastic components of the algorithms. This results in a total of $780$ problem instances solved for each algorithm.
All algorithms tested in this section are implemented within the \nomad software. They share common search components, including a quadratic model-based search, a speculative search, and a Nelder-Mead search. 

\begin{table}[htb!]
\centering
\begin{tabular}{|l|c|c|c||l|c|c|c|}
\hline
Problem & $n$ & $m$ & $\#\,x_0$ & 
Problem & $n$ & $m$ & $\#\,x_0$ \\
\hline \hline
CHENWANG\_F2 &  8 &  6 & 11 & MEZMONTES    &  2 &  2 &  1 \\
CHENWANG\_F3 & 10 &  8 & 10 & OPTENG\_RBF  &  3 &  4 &  1 \\
CRESCENT     & 10 &  2 &  1 & PENTAGON     &  6 & 15 &  1 \\
DISK         & 10 &  1 &  1 & SNAKE        &  2 &  2 &  1 \\
G210         & 10 &  2 &  1 & SPRING       &  3 &  4 & 10 \\
MAD6         &  5 &  7 & 10 & TAOWANG\_F2  &  7 &  4 & 10 \\
MDO          & 10 & 10 & 10 & ZHAOWANG\_F5 & 13 &  9 & 10 \\
\hline
\end{tabular}
\caption{A set of analytical constrained blackbox optimization problems~\cite[Table~1]{AuBrDiLeDSiTr26}.}
\label{tab:problems}
\end{table}

Two variants of \adspb are considered, differing by the choice of poll type. These variants are compared with the corresponding versions of \madspb using the same poll configurations. 
The \texttt{$2n$} variant generates a positive spanning set composed of $2n$ orthogonal directions in $\R^n$, as described in~\cite{AbAuDeLe09}. In contrast, the \texttt{$n+1$}~\cite{AuIaLeDTr2014} variant constructs a set of $n$ orthogonal directions and completes it with an additional direction obtained from a quadratic model, yielding a positive spanning set of size $n+1$. 
In particular, the \madspb $n+1$ variant corresponds to the current default implementation in \nomad. 
For \adspb, the poll directions are used directly to generate trial points, without any projection onto a mesh.

\Cref{fig:analytical} presents the data profiles obtained on the set of analytical problems from the literature, for three target accuracies $\tau \in \{ 10^{-3}, 10^{-5}, 10^{-7}\}$. As the tolerance $\tau$ decreases, the convergence criterion becomes more strict, requiring higher solution accuracy.
Overall, the \adspb variants consistently outperform their \madspb counterparts across all tolerances. For each poll strategy, the \adspb curves dominate those of \madspb.
A similar trend is observed when comparing poll strategies. The \texttt{$n+1$} variants systematically outperform the \texttt{$2n$} variants for both \adspb and \madspb. This improvement can be attributed to the use of a quadratic model to generate the additional polling direction, which enhances the quality of the directions. Moreover, the \texttt{$n+1$} strategy requires at most $n+1$ polling directions, compared to $2n$ for the \texttt{$2n$} variant, leading to a reduced number of function evaluations per iteration. This more economical use of evaluations further contributes to its better performance.
The effect is particularly pronounced at the restrictive tolerance $\tau = 10^{-7}$, where the gap between the two \adspb variants becomes more noticeable than for \madspb. 
One possible explanation is that \adspb benefits more from the model-based direction, since this direction is evaluated directly rather than projected onto a mesh, allowing the algorithm to make better use of high-quality trial points.
Finally, it is worth emphasizing that \adspb outperforms the \madspb \texttt{$n+1$} variant, which corresponds to the current default implementation in \nomad.

\begin{figure}[htb!]
\centering

\begin{subfigure}[t]{0.333\textwidth}
    \centering
    \includegraphics[width=\linewidth]{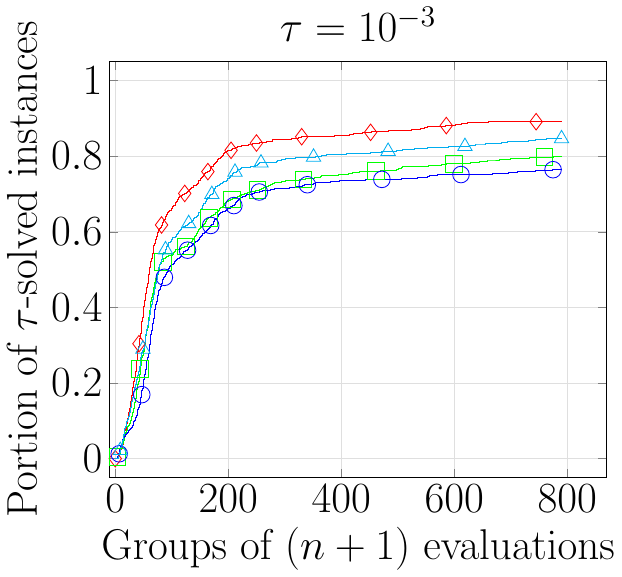}
\end{subfigure}
\hfill
\begin{subfigure}[t]{0.31\textwidth}
    \centering
    \includegraphics[width=\linewidth]{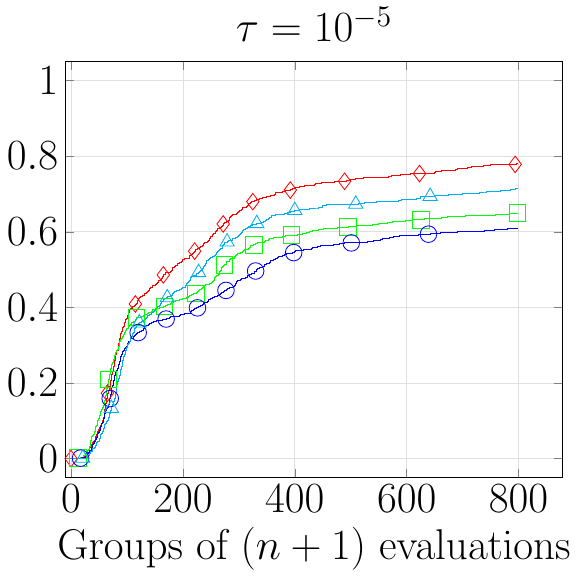}
\end{subfigure}
\hfill
\begin{subfigure}[t]{0.31\textwidth}
    \centering
    \includegraphics[width=\linewidth]{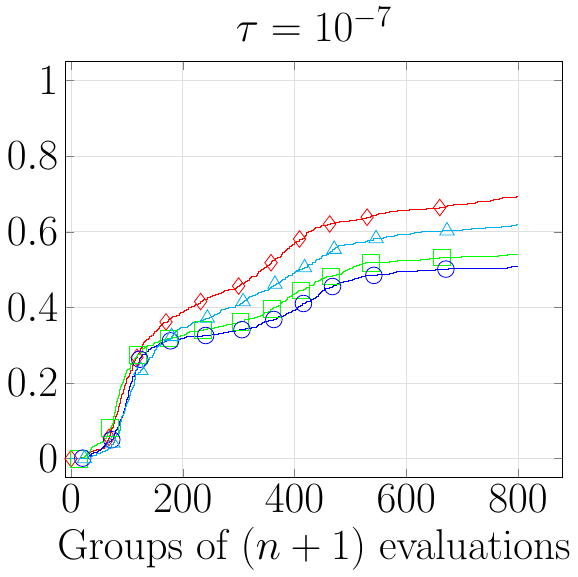}
\end{subfigure}

\vspace{0.8em}

\begin{tikzpicture}
\begin{axis}[
    hide axis,
    xmin=0, xmax=1,
    ymin=0, ymax=1,
    legend columns=4,
    legend style={
        font=\fontsize{10}{11}\selectfont,
        legend cell align=left,
        cells={anchor=west},
        draw=black,
        column sep=0.8em,
        inner xsep=0.4em,
        inner ysep=0.3em
    }
]
  \addlegendimage{line legend, solid, color=red, mark=diamond, mark size=2.8pt}
  \addlegendentry{ADS-PB \texttt{$n+1$}}
  \addlegendimage{line legend, solid, color=green, mark=square, mark size=2.8pt}
  \addlegendentry{MADS-PB \texttt{$n+1$}}
  \addlegendimage{line legend, solid, color=cyan, mark=triangle, mark size=2.8pt}
  \addlegendentry{ADS-PB \texttt{$2n$}}
  \addlegendimage{line legend, solid, color=blue, mark=o, mark size=2.8pt}
  \addlegendentry{MADS-PB \texttt{$2n$}}
\end{axis}
\end{tikzpicture}
\caption{Data profiles on the problems of \cref{tab:problems}. }
\label{fig:analytical}
\end{figure}

\subsection{Results on blackbox problems}

The results from the previous section were generated on a large number of analytical test problem instances.
The present section focuses on time-consuming blackbox optimization problems.
These problems are representative of realistic blackbox settings, where function evaluations are costly and analytical expressions are unavailable. They allow assessing the behavior of the algorithms in more challenging and practical scenarios.
The experiments are conducted using \adspb and \madspb within \nomad, with default configurations combining a quadratic model-based search, a speculative search, a Nelder-Mead search, and an \texttt{$n+1$} poll. These methods are compared with external algorithms from the literature: NSGA-II, implemented in Python through the \texttt{pymoo} library; CMA-ES, implemented using the Python package \texttt{cma}; and LOGDS, using the {\sf MATLAB} implementation made available by the authors of~\cite{BrCuLiSi2024}.

The first problem, referred to as \styrene, is introduced in~\cite{AuBeLe08}. It corresponds to the optimization of a simulated styrene production process which involves several stages, including reactants preparation, catalytic reactions, and separation steps through distillation, with the recycling of unreacted components. The simulation follows a sequential modular approach, where each unit operation is evaluated in sequence, leading to complex dependencies between variables. The resulting optimization problem is constrained, with $n=8$ variables and $m=11$ inequality constraints, including four non relaxable constraints.
To account for variability and robustness, this problem is solved from $30$ different feasible initial points\footnote{Starting points from \url{https://github.com/bbopt/styrene}} and using $3$ different random seeds.

The second problem is derived from the solar thermal power plant simulator introduced in~\cite{solar_paper},
which is designed as a benchmark for blackbox optimization. The instance considered here is \solar{6}, corresponding to a constrained cost minimization problem with $n=5$ variables and $m=6$ inequality constraints. 
Due to the high computational cost of function evaluations, only $30$ problem instances are considered for each algorithm. These correspond to $30$ initial points generated via Latin hypercube sampling\footnote{Starting points from \url{https://github.com/bbopt/solar}}, as provided in the associated benchmark, and a single random seed is used. The overall computational effort for the \solar{6} runs is approximately four CPU-days.

\Cref{fig:blackbox_profiles} displays the data profiles obtained on the \styrene and \solar{6} problems, respectively, for tolerances $\tau \in \{ 10^{-2}, 10^{-3}, 10^{-4}\}$. These tolerances are less strict than those considered in the previous section, reflecting the higher difficulty and computational cost of \solar{6}.

\begin{figure}[htb!]
\centering

\begin{subfigure}[t]{0.333\textwidth}
    \centering
    \includegraphics[width=\linewidth]{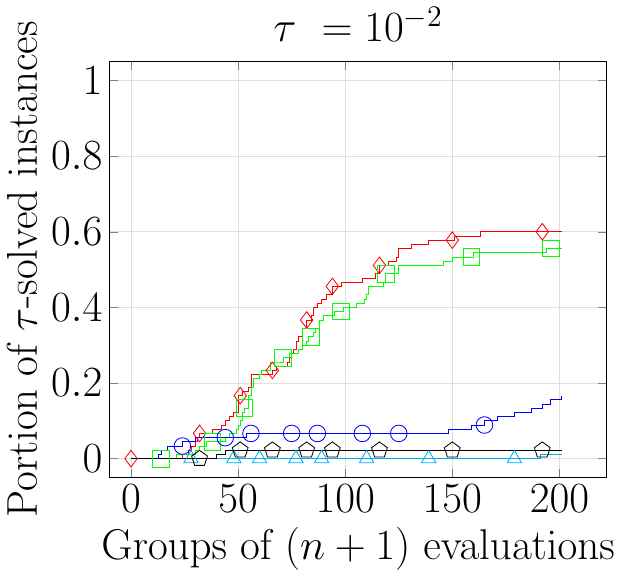}
\end{subfigure}
\hfill
\begin{subfigure}[t]{0.31\textwidth}
    \centering
    \includegraphics[width=\linewidth]{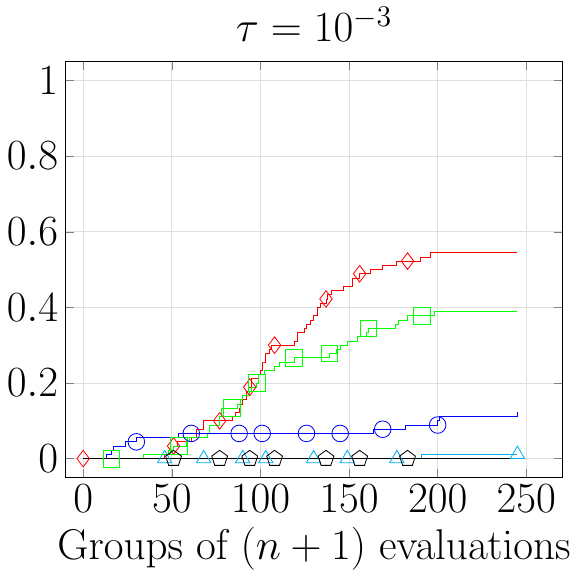}
\end{subfigure}
\hfill
\begin{subfigure}[t]{0.31\textwidth}
    \centering
    \includegraphics[width=\linewidth]{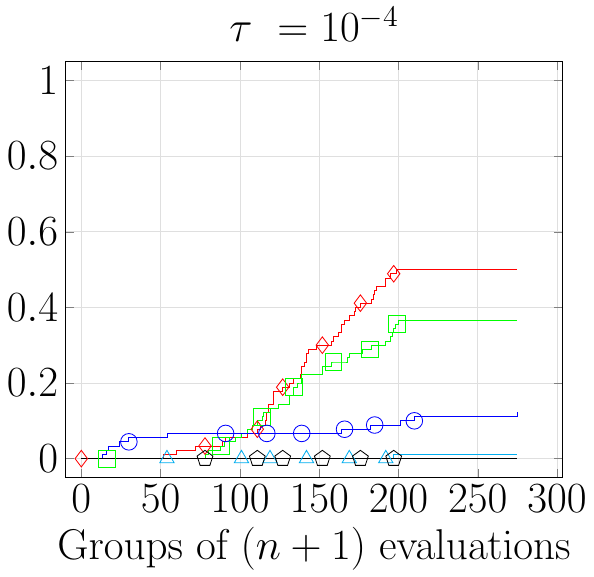}
\end{subfigure}

\vspace{0.4em}

{\small (a)~ \styrene }

\vspace{1em}

\begin{subfigure}[t]{0.333\textwidth}
    \centering
    \includegraphics[width=\linewidth]{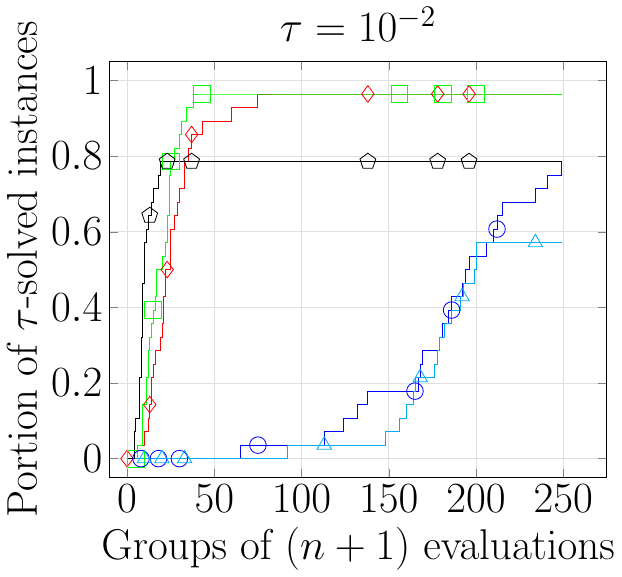}
\end{subfigure}
\hfill
\begin{subfigure}[t]{0.31\textwidth}
    \centering
    \includegraphics[width=\linewidth]{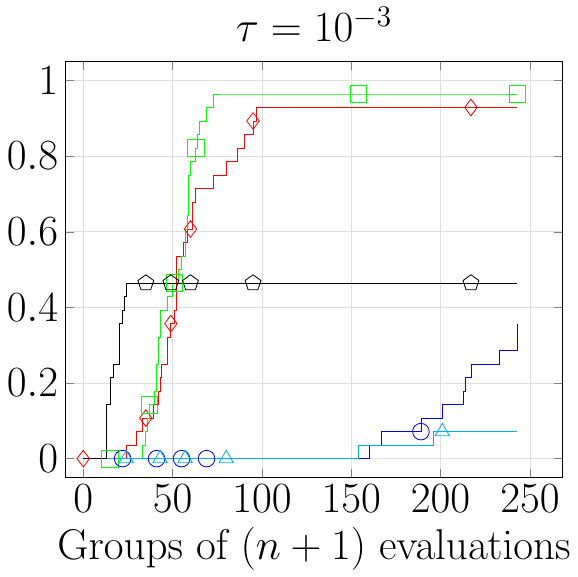}
\end{subfigure}
\hfill
\begin{subfigure}[t]{0.31\textwidth}
    \centering
    \includegraphics[width=\linewidth]{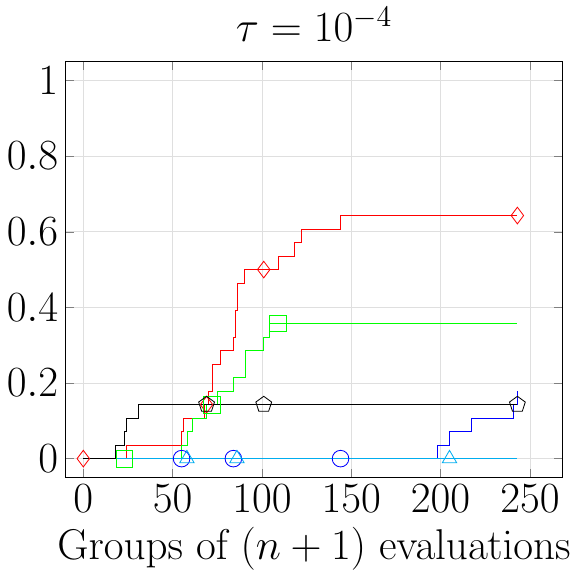}
\end{subfigure}

\vspace{0.4em}

{\small (b)~ \solar{6}}

\vspace{0.8em}

\begin{tikzpicture}
\begin{axis}[
    hide axis,
    xmin=0, xmax=1,
    ymin=0, ymax=1,
    legend columns=5,
    legend style={
        font=\fontsize{10}{12}\selectfont,
        legend cell align=left,
        cells={anchor=west},
        draw=black,
        column sep=1.2em
    }
]
  \addlegendimage{line legend, solid, color=red, mark=diamond, mark size=3.5pt}
  \addlegendentry{ADS-PB}
  \addlegendimage{line legend, solid, color=green, mark=square, mark size=3.5pt}
  \addlegendentry{MADS-PB}
\addlegendimage{line legend, solid, color=cyan, mark=triangle, mark size=3.5pt}
  \addlegendentry{NSGA-II}
  \addlegendimage{line legend, solid, color=blue, mark=o, mark size=3.5pt}
  \addlegendentry{CMA-ES}
  \addlegendimage{line legend, solid, color=black, mark=pentagon, mark size=4pt}
  \addlegendentry{LOGDS}
\end{axis}
\end{tikzpicture}

\caption{Data profiles on the \styrene and \solar{6} problems.}
\label{fig:blackbox_profiles}
\end{figure}

On the \styrene problem, \adspb consistently outperforms \madspb across all tolerances, $\tau$-solving more problems within a given computational budget. 
The relatively poor performance of LOGDS on this problem is explained by the fact that the four non quantifiable constraints are often violated, and blackbox evaluations often fail due to hidden constraints. 
Since LOGDS is not designed to handle non-relaxable and non-quantifiable constraints, infeasible evaluations provide little to no useful information, which significantly degrades its performance. 
The performance of CMA-ES and NSGA-II is also markedly lower, highlighting the difficulty of handling such constrained blackbox settings with general-purpose evolutionary methods. In contrast, \madspb remains competitive for moderate tolerances ($\tau = 10^{-2}$ and $10^{-3}$), but a clear gap emerges at stricter tolerances ($\tau = 10^{-4}$), where \adspb shows a strong advantage. 

On \solar{6}, \adspb and \madspb again dominate the other methods, although the performance gap is less pronounced. One explanation is that LOGDS does not benefit from the full range of search and poll mechanisms available in \nomad, such as the efficient $n+1$ poll using quadratic models to form the directions. 
Additionally, LOGDS appears to converge prematurely, which may limit its ability to fully exploit the available computational budget and refine solution quality. 
As with \styrene, CMA-ES and NSGA-II perform significantly worse, confirming the challenge posed by these constrained blackbox problems. 
Overall, while \madspb remains competitive at moderate tolerances, \adspb maintains a consistent advantage when higher accuracy is required.

This improved performance of \adspb over \madspb may be attributed to the effective use of quadratic models that exploit locally smooth behaviour even when the global problem is nonsmooth, especially since search points are evaluated directly without projecting on the mesh.

\section{Discussion}
\label{sec-discussion}

This work introduces \adspb, an extension of \ads to constrained blackbox optimization problems with quantifiable and relaxable constraints. 
By combining the PB mechanism with \ads, the proposed framework balances feasibility and objective function decrease while avoiding mesh projections and sufficient decrease conditions.
The implementation within \nomad enables a direct comparison with \madspb under identical algorithmic components. 

The computational results indicate that \adspb can achieve improved performance, especially at higher accuracy levels, likely due to its ability to directly evaluate high-quality trial points. 
Beyond these results, \adspb provides a flexible structure that can accommodate model-based and hybrid strategies while maintaining convergence guaranties under standard assumptions. 

Future work will investigate hybrid strategies that combine ADS with sufficient decrease mechanisms to possibly derive complexity guarantees. 
Moreover, another direction is to develop a fusion of MADS-PIP~\cite{AuBrDiLeDSiTr26} and \adspb, with the aim of efficiently leveraging both relaxable equality and inequality constraints. Another promising direction is the design of new search strategies that explicitly exploit the absence of mesh projection in \adspb, together with an analysis of their theoretical impact. 


\paragraph{Acknowledments}
This research is funded by the NSERC discovery grants  
    RGPIN-2026-05340 (Audet),
    RGPIN-2024-05093 (Diouane)
    and RGPIN-2024-05086 (Le~Digabel).
\paragraph{Conflict of interests}
The authors declare that they have no conflicts of interest.

\paragraph{Data availability}
Not applicable.

\paragraph{Use of AI}
The authors used ChatGPT to assist with figure layout and to improve English formulations.
ChatGPT was not used in any way for the scientific content.


\bibliographystyle{plainurl}
\small
\bibliography{bibliography}

@book{AuHa2026,
    address	= {Cham, Switzerland},
	author		= {C. Audet and W. Hare},
	Title		= {{Derivative-Free and Blackbox Optimization}},
	Series		= {{Springer Series in Operations Research and Financial Engineering}},
	Year		= {2026},
	Edition		= {Second},
	Publisher 	= {Springer International Publishing},
	Pages		= {425},
	DOI 		= {10.1007/978-3-032-00906-7}
}

@book{CoScVibook,
  address       = {Philadelphia},
  author        = {A.R. Conn and K. Scheinberg and L.N. Vicente},
  doi           = {10.1137/1.9780898718768},
  isbn          = {978-0-898716-68-9},
  publisher     = {SIAM},
  series        = {MOS-SIAM Series on Optimization},
  title         = {{Introduction to Derivative-Free Optimization}},
  url           = {https://doi.org/10.1137/1.9780898718768},
  year          = {2009}
}

@book{Clar83a,
  address   = {New York},
  author    = {F.H. Clarke},
  note      = {Reissued in 1990 by SIAM Publications, Philadelphia,
               as Vol.~5 in the series Classics in Applied Mathematics},
  publisher = {John Wiley and Sons},
  title     = {{Optimization and Nonsmooth Analysis}},
  url       = {http://www.ec-securehost.com/SIAM/CL05.html},
  year      = {1983}
}

@book{Jahn94a,
  Author    = {J. Jahn},
  Title     = {{Introduction to the Theory of Nonlinear Optimization}},
  Publisher = {Springer Cham},
  Edition   = {4th},
  Year      = {2020},
  Doi       = {10.1007/978-3-030-42760-3},
  Url       = {https://doi.org/10.1007/978-3-030-42760-3}
}

@article{AuTr2018,
  Author        = {C. Audet and C. Tribes},
  Title         = {{Mesh-based Nelder-Mead algorithm for inequality constrained optimization}},
  Journal       = {Computational Optimization and Applications},
  Volume        = {71},
  Number        = {2},
  Pages         = {331--352},
  Year          = {2018},
  Doi           = {10.1007/s10589-018-0016-0},
  Url           = {https://doi.org/10.1007/s10589-018-0016-0},
}

@article{BrCuLiSi2024,
  Author  = {A. Brilli and A.L. Custódio and G. Liuzzi and E.J. Silva},
  Title   = {{Nonlinear derivative-free constrained optimization with a penalty-interior point method and direct search}},
  Journal = {Numerical Algorithms},
  Year    = {2026},
  Doi     = {10.1007/s11075-026-02360-5},
  Url     = {https://doi.org/10.1007/s11075-026-02360-5}
}

@article{cmaes,
  Author        = {N. Hansen and A. Ostermeier},
  Title         = {{Completely derandomized self-adaptation in evolution strategies}},
  Journal       = {Evolutionary computation},
  Volume        = {9},
  Number        = {2},
  Pages         = {159--195},
  Year          = {2001},
  Doi           = {10.1162/106365601750190398},
  Url           = {https://doi.org/10.1162/106365601750190398},
}

@techreport{LiuLi2025,
  Author        = {Y. Liu and Y. Li},
  Title         = {{A Model-Based Derivative-Free Optimization Algorithm for Partially Separable Problems}},
  Institution   = {arXiv},
  Number        = {2506.21948},
  Year          = {2025},
  Arxivurl      = {https://doi.org/10.48550/arXiv.2506.21948},
  Url           = {https://doi.org/10.48550/arXiv.2506.21948}
}

@article{Gupta2024,
  Author  = {R. Gupta and Q. Zhang},
  Title   = {{Data-driven decision-focused surrogate modeling}},
  Journal = {AIChE Journal},
  Volume  = {70},
  Number  = {4},
  Pages   = {e18338},
  Year    = {2024},
  Doi     = {10.1002/aic.18338},
  Url     = {https://doi.org/10.1002/aic.18338}
}

@article{Ma2025,
  Author  = {K. Ma and L.M. Rios and H. Zheng and N.V. Sahinidis and S. Rajagopalan},
  Title   = {{Model-and-search: a derivative-free local optimization algorithm}},
  Journal = {Computational Optimization and Applications},
  Volume  = {92},
  Number  = {3},
  Pages   = {889--921},
  Year    = {2025},
  Doi     = {10.1007/s10589-025-00686-9},
  Url     = {https://doi.org/10.1007/s10589-025-00686-9}
}

@article{Karantoumanis2024,
  Author  = {E. Karantoumanis and N. Ploskas},
  Title   = {{Improving derivative-free optimization algorithms through an adaptive sampling procedure}},
  Journal = {Results in Control and Optimization},
  Volume  = {16},
  Pages   = {100460},
  Year    = {2024},
  Doi     = {10.1016/j.rico.2024.100460},
  Url     = {https://doi.org/10.1016/j.rico.2024.100460}
}

@article{Cartis2026,
  Author  = {C. Cartis and L. Roberts},
  Title   = {{Randomized subspace derivative-free optimization with quadratic models and second-order convergence}},
  Journal = {Optimization Methods and Software},
  Pages   = {1--28},
  Year    = {2026},
  Doi     = {10.1080/10556788.2025.2601668},
  Url     = {https://doi.org/10.1080/10556788.2025.2601668}
}

@techreport{AuBrDiLeDSiTr26,
  Author      = {C. Audet and A. Brilli and Y. Diouane and S. {Le~Digabel} and E.J. Silva and C. Tribes},
  Title       = {{A penalty-interior point method combined with MADS  for equality and inequality constrained optimization}},
  Institution = {arXiv},
  Number      = {2601.20811},
  Year        = {2026},
  Arxivurl    = {https://doi.org/10.48550/arXiv.2601.20811},
  Url         = {https://doi.org/10.48550/arXiv.2601.20811}
}

@techreport{denorme-ads-2025,
  Author        = {C. Audet and T. Denorme and Y. Diouane and S. {Le~Digabel} and C. Tribes},
  Title         = {{Adaptive direct search algorithms for constrained optimization}},
  Institution   = {Les cahiers du GERAD},
  Number        = {G-2025-53},
  Year          = {2025},
  Arxivurl      = {https://arxiv.org/abs/2507.23054},
  Doi           = {10.48550/arXiv.2507.23054}
}

@article{solar_paper,
  Author  = {N. Andr\'{e}s-Thi\'{o} and C. Audet and M. Diago and A.E. Gheribi and S. {Le~Digabel} and X. Lebeuf and M. {Lemyre~Garneau} and C. Tribes},
  Title   = {{Solar: a solar thermal power plant simulator for blackbox optimization benchmarking}},
  Journal = {Optimization and Engineering},
  Volume  = {26},
  Number  = {3},
  Pages   = {1815--1861},
  Doi     = {10.1007/s11081-024-09952-x},  
  Url     = {https://doi.org/10.1007/s11081-024-09952-x},
  Year    = {2025}
}

@article{DzRiRoZe2025,
  Author  = {K.J. Dzahini and F. Rinaldi and C.W. Royer and D. Zeffiro},
  Title   = {{Direct-search methods in the year 2025: Theoretical guarantees and algorithmic paradigms}},
  Journal = {EURO Journal on Computational Optimization},
  Volume  = {13},
  Pages   = {100110},
  Year    = {2025},
  Doi     = {10.1016/j.ejco.2025.100110},
  Url     = {https://doi.org/10.1016/j.ejco.2025.100110}
}

@article{BeSoVi2023,
  author        = {A.S. Berahas and O. Sohab and  L.N. Vicente},
  doi           = {10.1080/10556788.2022.2142582},
  journal       = {Optimization Methods and Software},
  number        = {2},
  pages         = {386--411},
  title         = {{Full-low evaluation methods for derivative-free optimization}},
  url           = {https://doi.org/10.1080/10556788.2022.2142582},
  volume        = {38},
  year          = {2023}
}

@article{G-2025-36,
  Author        = {C. Audet and W. Hare and C. Tribes},
  Title         = {{A summary of benchmarking constrained, multi-objective and surrogate-assisted optimization methods}},
  Journal       = {Optimization Letters},
  Volume        = {},
  Number        = {},
  Pages         = {},
  Year          = {2026},
  Doi           = {10.1007/s11590-026-02302-z},
  Url           = {https://doi.org/10.1007/s11590-026-02302-z},
}

@article{Rock80a,
  author        = {R.T. Rockafellar},
  journal       = {Canadian Journal of Mathematics},
  number        = {2},
  pages         = {257--280},
  title         = {{Generalized directional derivatives and subgradients of nonconvex functions}},
  volume        = {32},
  year          = {1980}
}

@article{LedWild2015,
  author        = {S. {Le~Digabel} and S.M. Wild},
  doi           = {10.1007/s11081-023-09839-3},
  journal       = {Optimization and Engineering},
  number        = {2},
  pages         = {1125--1143},
  title         = {{A taxonomy of constraints in black-box simulation-based optimization}},
  url           = {https://doi.org/10.1007/s11081-023-09839-3},
  volume        = {25},
  year          = {2024}
}

@article{DIOUANE2021100001,
  author        = {Y. Diouane},
  doi           = {10.1016/j.ejco.2020.100001},
  journal       = {EURO Journal on Computational Optimization},
  pages         = {100001},
  title         = {{A merit function approach for evolution strategies}},
  url           = {https://doi.org/10.1016/j.ejco.2020.100001},
  volume        = {9},
  year          = {2021}
}

@article{diouane2022trego,
  author        = {Y. Diouane and V. Picheny and R. {Le~Riche} and A. {Scotto~Di~Perrotolo}},
  doi           = {10.1007/s10898-022-01245-w},
  journal       = {Journal of Global Optimization},
  pages         = {1--23},
  title         = {{TREGO: a trust-region framework for efficient global optimization}},
  url           = {https://doi.org/10.1007/s10898-022-01245-w},
  year          = {2022}
}

@article{nomad4paper,
  author        = {C. Audet and S. {Le~Digabel} and V. {Rochon~Montplaisir} and C. Tribes},
  doi           = {10.1145/3544489},
  journal       = {{ACM} Transactions on Mathematical Software},
  number        = {3},
  pages         = {35:1--35:22},
  title         = {{Algorithm~1027: NOMAD version~4: Nonlinear optimization with the MADS algorithm}},
  url           = {https://doi.org/10.1145/3544489},
  volume        = {48},
  year          = {2022}
}

@article{CoLed2011,
  author        = {A.R. Conn and S. {Le~Digabel}},
  doi           = {10.1080/10556788.2011.623162},
  journal       = {Optimization Methods and Software},
  number        = {1},
  pages         = {139--158},
  title         = {{Use of quadratic models with mesh-adaptive direct search for constrained black box optimization}},
  url           = {https://doi.org/10.1080/10556788.2011.623162},
  volume        = {28},
  year          = {2013}
}

@article{AlAuGhKoLed2020,
  author        = {S. Alarie and C. Audet and A.E. Gheribi and M. Kokkolaras and S. {Le~Digabel}},
  doi           = {10.1016/j.ejco.2021.100011},
  journal       = {EURO Journal on Computational Optimization},
  pages         = {100011},
  title         = {{Two decades of blackbox optimization applications}},
  url           = {https://doi.org/10.1016/j.ejco.2021.100011},
  volume        = {9},
  year          = {2021}
}

@inproceedings{Deb2000,
  address       = {Berlin, Heidelberg},
  author        = {K. Deb and S. Agrawal and A. Pratap and T. Meyarivan},
  booktitle     = {{Parallel Problem Solving from Nature PPSN VI}},
  doi           = {10.1007/3-540-45356-3_83},
  editor        = {M. Schoenauer and K. Deb and G. Rudolph and X. Yao and E. Lutton and J.J. Merelo and H.P. Schwefel, Hans-Paul},
  pages         = {849--858},
  publisher     = {Springer},
  title         = {{A Fast Elitist Non-dominated Sorting Genetic Algorithm for Multi-objective Optimization: NSGA-II}},
  url           = {https://doi.org/10.1007/3-540-45356-3_83},
  year          = {2000}
}

@article{MoWi2009,
  author        = {J.J. Mor\'e and S.M. Wild},
  doi           = {10.1137/080724083},
  journal       = {SIAM Journal on Optimization},
  number        = {1},
  pages         = {172--191},
  title         = {{Benchmarking Derivative-Free Optimization Algorithms}},
  url           = {https://doi.org/10.1137/080724083},
  volume        = {20},
  year          = {2009}
}

@article{LaMeWi2019,
  author        = {J. Larson and M. Menickelly and S.M. Wild},
  doi           = {10.1017/S0962492919000060},
  journal       = {Acta Numerica},
  pages         = {287--404},
  title         = {{Derivative-free optimization methods}},
  url           = {https://doi.org/10.1017/S0962492919000060},
  volume        = {28},
  year          = {2019}
}

@article{AuBeLe08,
  author        = {C. Audet and V. B\'echard and S. {Le~Digabel}},
  doi           = {10.1007/s10898-007-9234-1},
  journal       = {Journal of Global Optimization},
  number        = {2},
  pages         = {299--318},
  title         = {{Nonsmooth optimization through Mesh Adaptive Direct Search
                   and Variable Neighborhood Search}},
  url           = {https://doi.org/10.1007/s10898-007-9234-1},
  volume        = {41},
  year          = {2008}
}

@article{AuDe09a,
  author        = {C. Audet and J.E. {Dennis, Jr.}},
  doi           = {10.1137/070692662},
  journal       = {SIAM Journal on Optimization},
  number        = {1},
  pages         = {445--472},
  title         = {{A Progressive Barrier for Derivative-Free Nonlinear Programming}},
  url           = {https://doi.org/10.1137/070692662},
  volume        = {20},
  year          = {2009}
}

@article{AuDe2006,
  author        = {C. Audet and J.E. {Dennis, Jr.}},
  doi           = {10.1137/040603371},
  journal       = {SIAM Journal on Optimization},
  number        = {1},
  pages         = {188--217},
  title         = {{Mesh Adaptive Direct Search Algorithms for Constrained Optimization}},
  url           = {https://doi.org/10.1137/040603371},
  volume        = {17},
  year          = {2006}
}

@incollection{CuScVi2017,
  address       = {Philadelphia},
  author        = {A.L. Cust{\'o}dio and K. Scheinberg and L.N. Vicente},
  booktitle     = {{Advances and Trends in Optimization with Engineering Applications}},
  chapter       = {37},
  editor        = {T. Terlaky and M.F. Anjos and S. Ahmed},
  publisher     = {SIAM},
  series        = {{MOS-SIAM Book Series on Optimization}},
  title         = {{Methodologies and software for derivative-free optimization}},
  url           = {http://www.mat.uc.pt/~lnv/papers/dfo-survey.pdf},
  year          = {2017}
}

@article{AuCoLedPey2016,
  author        = {C. Audet and A.R. Conn and S. {Le~Digabel} and M. Peyrega},
  doi           = {10.1007/s10589-018-0020-4},
  journal       = {Computational Optimization and Applications},
  number        = {2},
  pages         = {307--329},
  title         = {{A progressive barrier derivative-free trust-region algorithm for constrained optimization}},
  url           = {https://doi.org/10.1007/s10589-018-0020-4},
  volume        = {71},
  year          = {2018}
}

@article{GrVi2014,
  author        = {S. Gratton and L.N. Vicente},
  doi           = {10.1137/130917661},
  journal       = {SIAM Journal on Optimization},
  number        = {4},
  pages         = {1980--1998},
  title         = {{A Merit Function Approach for Direct Search}},
  url           = {https://doi.org/10.1137/130917661},
  volume        = {24},
  year          = {2014}
}

@article{KoLeTo03a,
  author        = {T.G. Kolda and R.M. Lewis and V. Torczon},
  doi           = {10.1137/S003614450242889},
  journal       = {SIAM Review},
  number        = {3},
  pages         = {385--482},
  title         = {Optimization by direct search: New perspectives on some classical and modern methods},
  url           = {https://doi.org/10.1137/S003614450242889},
  volume        = {45},
  year          = {2003}
}

@article{AuIaLeDTr2014,
  author        = {C. Audet and A. Ianni and S. {Le~Digabel} and C. Tribes},
  doi           = {10.1137/120895056},
  journal       = {SIAM Journal on Optimization},
  number        = {2},
  pages         = {621--642},
  title         = {{Reducing the Number of Function Evaluations in Mesh Adaptive Direct Search Algorithms}},
  url           = {https://doi.org/10.1137/120895056},
  volume        = {24},
  year          = {2014}
}

@article{AuDe03a,
  author        = {C. Audet and J.E. {Dennis, Jr.}},
  doi           = {10.1137/S1052623400378742},
  journal       = {SIAM Journal on Optimization},
  number        = {3},
  pages         = {889--903},
  title         = {{Analysis of Generalized Pattern Searches}},
  url           = {https://doi.org/10.1137/S1052623400378742},
  volume        = {13},
  year          = {2003}
}

@article{AbAuDeLe09,
  author        = {M.A. Abramson and C. Audet and J.E. {Dennis, Jr.} and S. {Le~Digabel}},
  doi           = {10.1137/080716980},
  journal       = {SIAM Journal on Optimization},
  number        = {2},
  pages         = {948--966},
  publisher     = {SIAM},
  title         = {{OrthoMADS: A Deterministic MADS Instance with Orthogonal Directions}},
  url           = {https://doi.org/10.1137/080716980},
  volume        = {20},
  year          = {2009}
}

@article{Davi54b,
  author  = {C. Davis},
  journal = {American Journal of Mathematics},
  pages   = {733--746},
  title   = {Theory of Positive Linear Dependence},
  url     = {http://www.ams.org/mathscinet-getitem?mr=16:211e},
  volume  = {76},
  year    = {1954}
}
\label{sec-refs}
\pdfbookmark[1]{References}{sec-refs}

\end{document}